\documentclass[11pt,twoside,a4paper]{article}
\usepackage[english]{babel}
\usepackage{a4wide}
\usepackage{graphicx}
\usepackage{amsmath,amssymb,amsthm, amsgen}
\usepackage{listings}
\usepackage{color}
\usepackage[usenames,dvipsnames]{xcolor}
\usepackage{rotating}
\usepackage{hhline}
\usepackage{multirow}
\usepackage{enumerate}
\usepackage{enumitem}
\usepackage[bookmarks]{}
\usepackage{amsfonts}   
\usepackage{dsfont}
\usepackage{tikz}
\usepackage{booktabs}
\usepackage{longtable}

\newtheorem{thm}{Theorem}[section]
\newtheorem{prop}[thm]{Proposition}
\newtheorem{lem}[thm]{Lemma}

\newtheorem{ass}[thm]{Assumption}

\newtheorem{rem}[thm]{Remark}

\def\XXint#1#2#3{{\setbox0=\hbox{$#1{#2#3}{\int}$ }
\vcenter{\hbox{$#2#3$ }}\kern-.6\wd0}}

\newcommand{\dist}{\operatorname{dist}}

\newcommand{\indicatornoacc}[1]{ \mathds{1}_{ #1 } }

\newcommand{\loc}{\text{loc}}

\newcommand{\N}{\mathbb N}
\newcommand{\NN}{E_n^{\operatorname{nn}}}

\newcommand{\R}{\mathbb R}

\newcommand{\supp}{\operatorname{supp}}
\newcommand{\Ta}{T_1}
\newcommand{\Tb}{T_2}
\newcommand{\Tc}{T_3}
\newcommand{\Td}{T_4}
\newcommand{\Te}{T_5}
\newcommand{\Tf}{T_6}
\newcommand{\Tg}{T_7}
\newcommand{\Th}{T_8}
\newcommand{\Ti}{T_9}

\newcommand{\Vext}{U}
\newcommand{\Vreg}{V_{ \operatorname{reg} }}

\newcommand{\xto}[1]{\xrightarrow{#1}}

\newcommand{\bx}{\mathbf x}
\newcommand{\by}{\mathbf y}

\newcommand{\cP}{\mathcal P}

\newcommand{\tphi}{\varphi^*}
\newcommand{\tell}{\ell^*}
\newcommand{\tm}{m^*}
\newcommand{\tx}{\bx^*}

\newcommand{\oell}{\overline \ell}

\newcommand{\ovm}{\overline m}
\newcommand{\ophi}{\overline \varphi}
\newcommand{\orho}{\overline \rho}
\newcommand{\ox}{\overline \bx}
\newcommand{\oox}{\overline x}

\DeclareFontFamily{U}{mathx}{\hyphenchar\font45}
\DeclareFontShape{U}{mathx}{m}{n}{
      <5> <6> <7> <8> <9> <10>
      <10.95> <12> <14.4> <17.28> <20.74> <24.88>
      mathx10
      }{}
\DeclareSymbolFont{mathx}{U}{mathx}{m}{n}
\DeclareFontSubstitution{U}{mathx}{m}{n}
\DeclareMathAccent{\widecheck}{0}{mathx}{"71}

\begin{document}

\title{Quantitative estimate of the continuum approximations of interacting particle systems in one dimension}

\author{M.~Kimura, P.~van Meurs}

\date{}

\maketitle 


\begin{abstract}
We consider a large class of interacting particle systems in 1D described by an  energy whose interaction potential is singular and non-local. This class covers Riesz gases (in particular, log gases) and applications to plasticity and approximation theory of functions. While it is well established that the minimisers of such interaction energies converge to a certain particle density profile as the number of particles tends to infinity, any bound on the rate of this convergence is only known in special cases by means of quantitative estimates. The main result of this paper extends these quantitative estimates to a large class of interaction energies by a different proof. The proof relies on one-dimensional features such as the convexity of the interaction potential and the ordering of the particles. The main novelty of the proof is the treatment of the singularity of the interaction potential by means of a carefully chosen renormalisation. 
\end{abstract}


\section{Introduction}

We are interested in quantifying the difference between minimisers of interacting particle energies and the minimisers of the related energies for the particle density. The interacting particle energies are given by
\begin{gather}\label{En}
    E_n (\bx) 
    = \frac1{n^2} \sum_{i=1}^n \sum_{ j=0 }^{i-1} V (x_i - x_j) + \frac1n \sum_{i=0}^{n} \Vext (x_i),
\end{gather}
where $n+1$ is the number of particles, and 
\begin{equation} \label{Omega}
  \bx 
  := (x_0, x_1, \ldots , x_n) 
  \in \Omega 
  := \{ \by \in \R^{n+1} : y_0 < y_1 < \ldots < y_n \}
\end{equation}
is the list of ordered particle positions. The energies $E_n$ are the sum of two parts. We interpret the first part as the interaction part, in which $V$ is the interaction potential, and the second part as a confinement term, in which $U$ is the confining potential. Typical examples of $V$ and $U$ are plotted in Figure \ref{fig:V}. We aim to keep the assumptions on $V$ and $U$ as weak as possible. These assumptions are as follows. 

\begin{figure}[h]
\centering
\begin{tikzpicture}[scale=1.5, >= latex]    
\def \w {2}
        
\draw[->] (0,-.4) -- (0,\w);
\draw[->] (-\w,0) -- (\w,0) node[right] {$x$};
\draw[domain=0.21:\w, smooth, very thick] plot (\x,{-ln(\x / 1.5)});
\draw[domain=-\w:-0.21, smooth, very thick] plot (\x,{{-ln(-\x / 1.5)}});
\draw (0.3, \w) node[anchor = north west]{$V(x)$};

\begin{scope}[shift={(2.5*\w,0)},scale=1]
\draw[->] (0,0) -- (0,\w);
\draw[->] (-\w,0) -- (\w,0) node[right] {$x$};
\draw[domain=-1.5:\w, smooth, very thick] plot (\x,{ 6*(4/(\x + 4) + \x/4 -1) });
\draw[dashed] (-1.5, 0) node[below]{$z_1$} -- (-1.5, \w);
\fill (-1.5, 1.35) circle (.05);
\draw (\w, .9) node[anchor = south east]{$\Vext(x)$};
\end{scope} 
\end{tikzpicture} \\
\caption{Typical examples of $V$ and $\Vext$.}
\label{fig:V}
\end{figure}
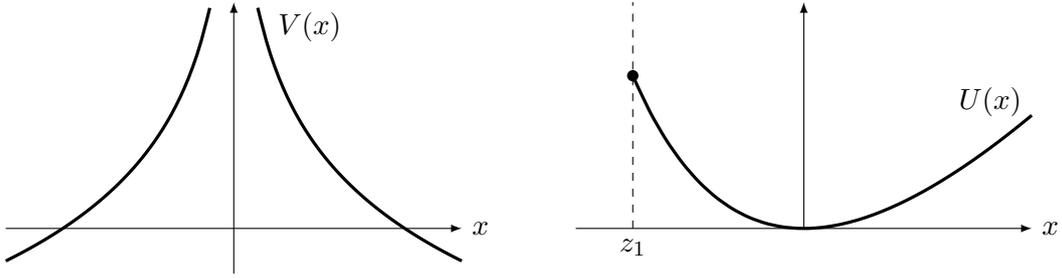 

\begin{ass}[$V$ and $U$] \label{ass:VU}
The interaction potential $V \in L_\loc^1 (\R)$ splits as $V = V_a + \Vreg$, where
\begin{equation} \label{Va}
  V_a (x) := \left\{ \begin{array}{ll}
  - \log |x|, \: & \text{if } a = 0 \\
  |x|^{-a}, & \text{if } 0 < a < 1,
  \end{array}  \right.
\end{equation} 
for a fixed parameter $a \in [0,1)$, and $\Vreg \in C^2(\R)$ is such that
\begin{gather} \label{V:props}
  V \text{ even}, \quad
  V \text{ convex on } (0,\infty), \quad
  \lim_{x \to \infty} \frac{V(x)}{x} = 0.
\end{gather} 

The domain of the confining potential $U : \R \to [0, \infty]$ is
\begin{equation*}
  D(U) := \{ x \in \R : U(x) < \infty \} = \overline{ (z_1, z_2) }
\end{equation*}
for some $-\infty \leq z_1 < z_2 \leq \infty$. It satisfies 
\begin{gather} \label{U:ass}
  U \in C^2 ( D(U) ), \quad
  U \text{ convex on } \R, \quad
  \min_{\R} U = 0, \quad
  \lim_{|x| \to \infty} U(x) = \infty. 
\end{gather}
\end{ass}

We interpret Assumption \ref{ass:VU} as follows. We consider $a$ as a fixed parameter which determines the singularity of $V$ at $0$. The part $\Vreg$ is a regular perturbation which determines the bulk and tails of $V$. Finite values for $z_i$ correspond to impenetrable barriers for the particle positions.
\smallskip

Given $E_n$, the related energy for the particle density $\rho$ is given by
\begin{equation} \label{E}
   E : \mathcal P (\R) \to \R \cup \{ +\infty \}, \qquad
   E (\rho) := \frac12 \int_\R \int_\R V(x-y) \, d\rho (y) \, d\rho (x) + \int_\R U(x) \, d\rho(x),
\end{equation}
where $\cP (\R)$ is the space of probability measures.

For various choices of $V$ and $U$, it is known 
(see, e.g., 
\cite{SaffTotik97,
GeersPeerlingsPeletierScardia13,
vanMeurs18,
VanMeursMunteanPeletier14}) 
that $E_n$ and $E$ attain their minimal value at some $\tx \in \Omega$ and $\orho \in \cP (\R)$ respectively, that $\orho \in \cP (\R)$ is unique, and that any sequence of minimisers $\tx$ (parametrised by $n$) converges to $\orho$ in a suitable topology as $n \to \infty$. Yet, any \emph{quantitative} estimate between $\tx$ and $\orho$ for finite $n$ is only available for special choices of $V$ and $U$ (see Section \ref{s:intro:disc}). The aim of this paper is to derive such an estimate for the much larger class of potentials $V$ and $U$ characterised in Assumption \ref{ass:VU}.
\smallskip

In order to give meaning to a quantitative estimate between $\tx$ and $\orho$, we construct from $\tx$ a probability density function $\tphi$, and seek to bound $\tphi - \orho$ in a suitable topology. More generally, for any $\bx \in \Omega$, we define a related density function $\varphi \in \cP(\R)$ given by the piecewise constant function
\begin{equation} \label{phi:from:x}
  \varphi (y) := \left\{ \begin{aligned}
    &\frac{1/n}{ x_i - x_{i-1} }
    &&\text{if } x_{i-1} < y < x_i \text{ for some } i \in \{1,\ldots,n\}  \\
    &0
    &&\text{otherwise.}
  \end{aligned} \right.
\end{equation} 
We denote by $\varphi^*$ the density function related to $\tx$. The choice of $\varphi$ is not unique; in Section \ref{s:intro:rems} we discuss a different choice based on Voronoi cells. The current choice in \eqref{phi:from:x} is made to ease computations, and has been used in earlier studies; see, e.g., \cite{HallChapmanOckendon10}.

Figure \ref{fig:tx} illustrates typical examples of $\tphi$ and $\orho$. Especially in the case where $D(U)$ confines $\tphi$, the graphs of $\tphi$ and $\orho$ are close to each other. This observation is in line with the literature (see, e.g.,
\cite{GarroniVanMeursPeletierScardia16,
GeersPeerlingsPeletierScardia13,
HallChapmanOckendon10,
HallHudsonVanMeurs18}). 
In this paper we wish to finally quantify this observation.

\begin{figure}[h]
\centering
\begin{tikzpicture}[scale=1.125]
\def \x {1.33}
\def \dx {.13}
\def \y {1.3}
\def \z {.07}
\begin{scope}[shift={(4.5,0)},scale=1]
  \node (label) at (\z,0){\includegraphics[height=3.6cm, trim=15mm 0mm 15mm 0mm]{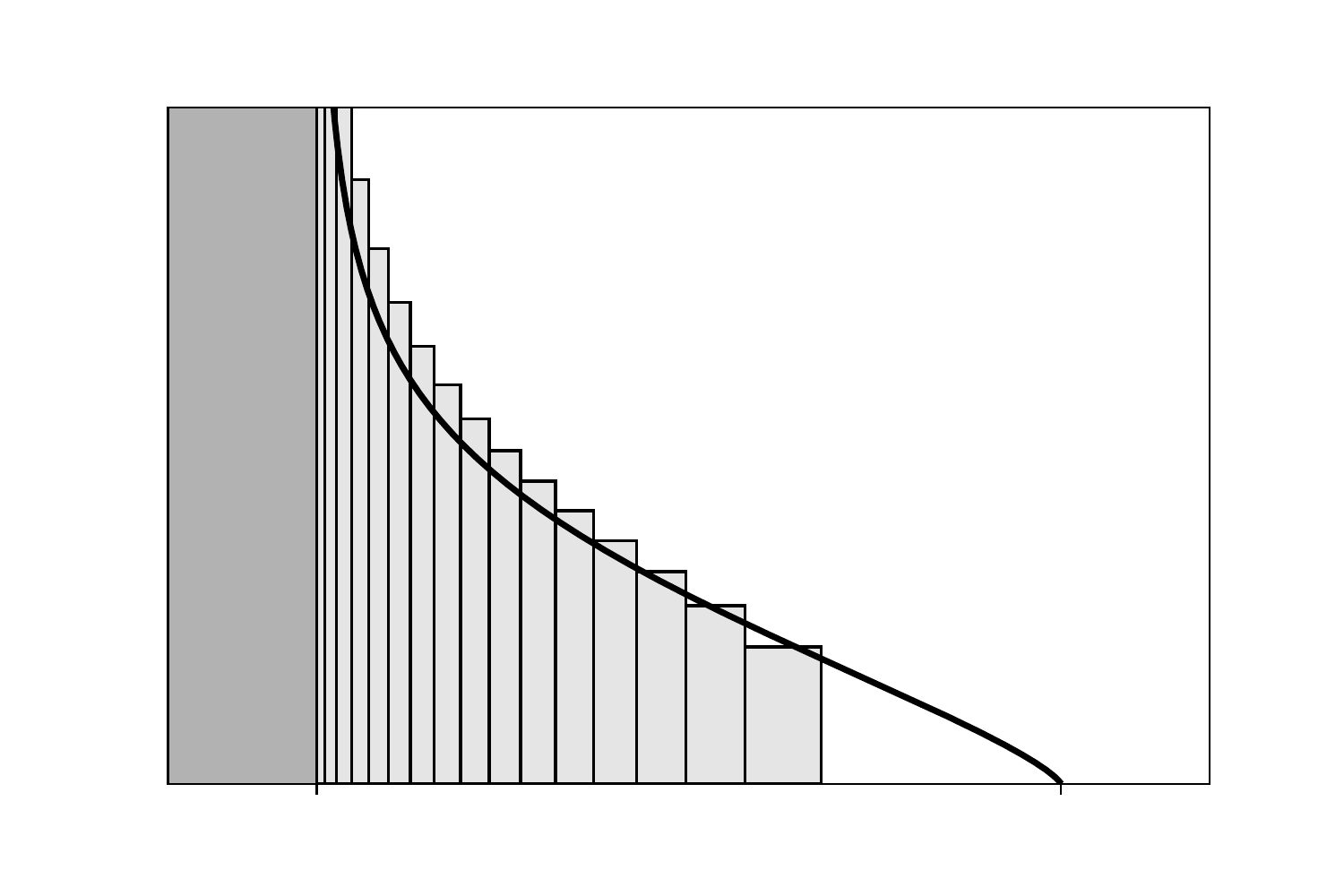}};
  \draw (-\x+\dx, -\y) node[below] {$0$};
  \draw (\x+\dx, -\y) node[below] {$1$};
  \draw (\dx,1.2) node[above] {$U(x) = \gamma_1 x$};
\end{scope}
\begin{scope}[shift={(0,0)},scale=1]
  \node (label) at (\z,0){\includegraphics[height=3.6cm, trim=15mm 0mm 15mm 0mm, clip]{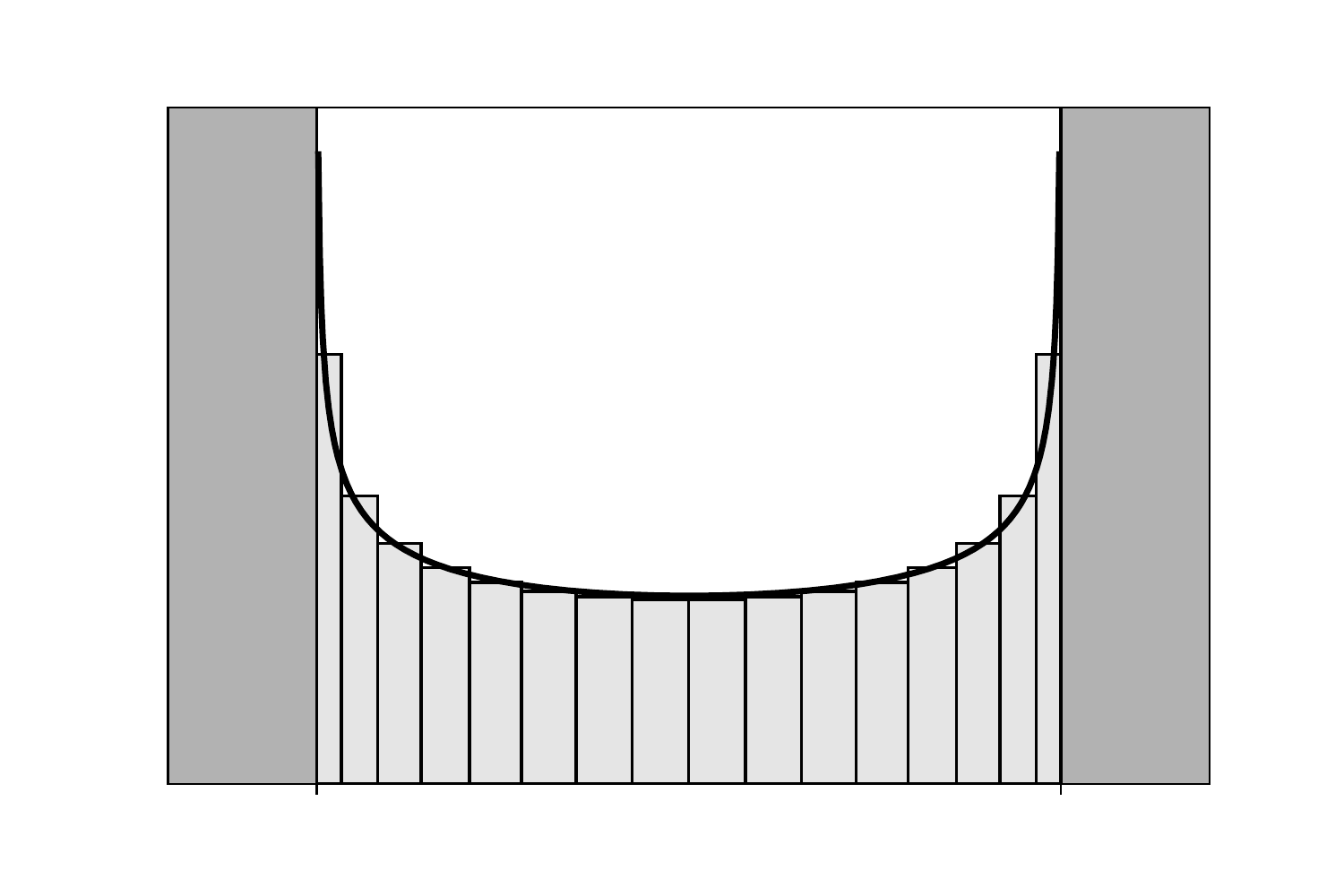}};
  \draw (-\x+\dx, -\y) node[below] {$0$};
  \draw (\x+\dx, -\y) node[below] {$1$};
  \draw (\dx,1.2) node[above] {$U(x) = 0$};
\end{scope}
\begin{scope}[shift={(9,0)},scale=1]
  \node (label) at (\z,0){\includegraphics[height=3.6cm, trim=15mm 0mm 15mm 0mm, clip]{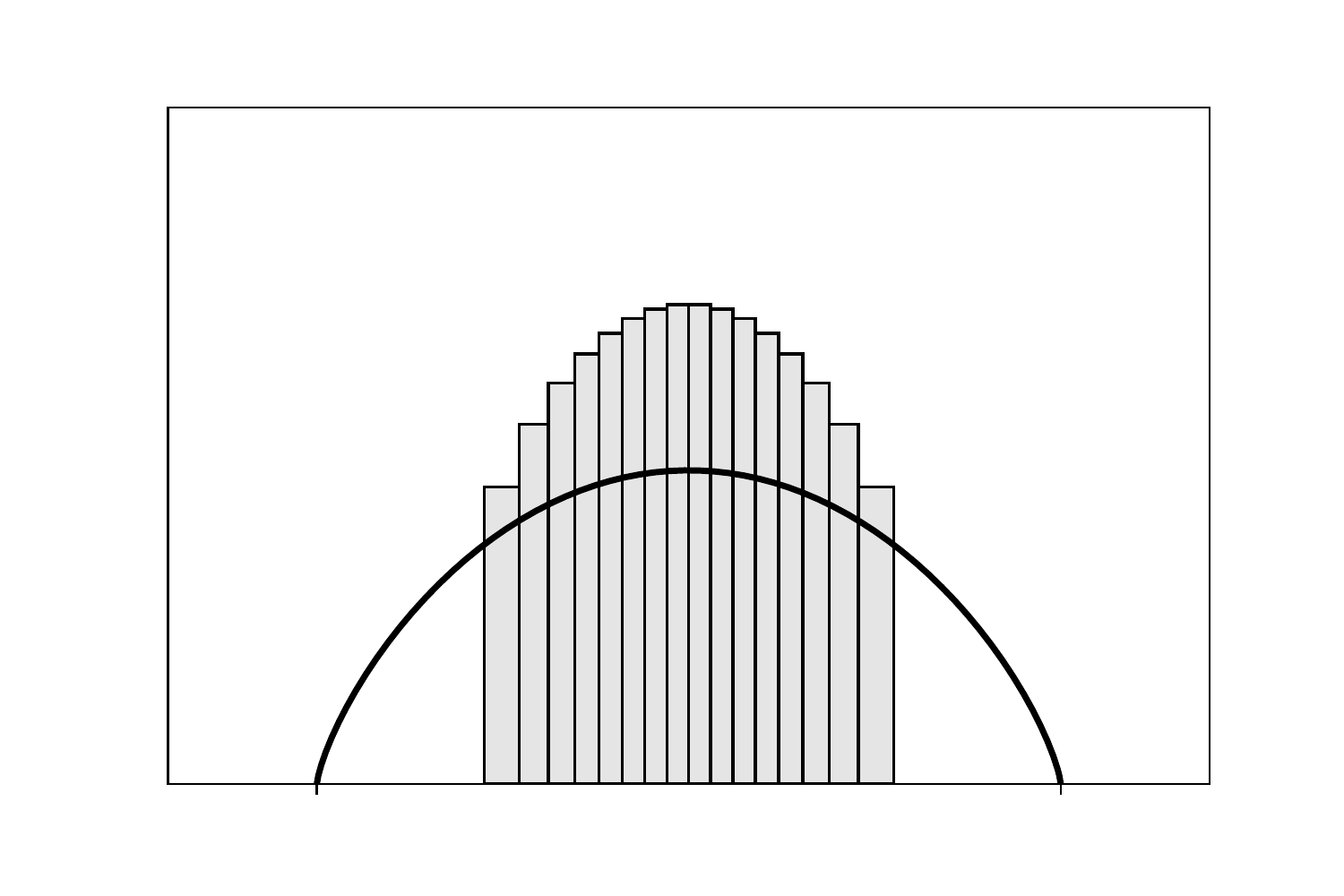}};
  \draw (-\x+\dx, -\y) node[below] {$0$};
  \draw (\x+\dx, -\y) node[below] {$1$};
  \draw (\dx,1.2) node[above] {$U(x) = \gamma_2 (x - \frac12)^2$};
\end{scope}
\end{tikzpicture} \\
\caption{Numerical computations of $\tx$ for $n = 16$, $V = V_a$ with $a = \frac12$ and three different choices of $U$. The points $x_i^*$ on the horizontal axis are indicated by the vertical edges of each light-gray rectangle (all with area $\frac1n$). The graph of $\tphi$ is given by the top edges of these rectangles. The black curve is the graph of $\orho$. The region where $U = \infty$ is indicated in gray. The values of $\gamma_i$ are chosen such that $\supp \orho = [0,1]$. The computation of $\tx$ and $\orho$ is explained in Section \ref{s:num}.}
\label{fig:tx}
\end{figure}
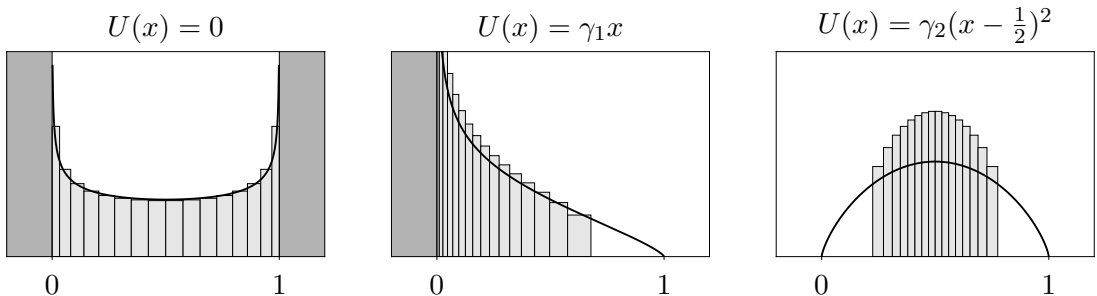

\subsection{Main result}
\label{s:intro:t}

We estimate $\tphi - \orho$ in terms of a fractional Sobolev norm. To introduce this norm, we define the fractional Sobolev space on $\R$ by
\begin{equation} \label{Hs}
  \begin{aligned}
    H^{-s} (\R) 
    &:= \big\{ \zeta \in \mathcal S'(\R) : {\textstyle \int_\R} (1 + \omega^2)^{-s} \big| \widehat \zeta (\omega) \big|^2 \, d\omega < \infty \big\} \\
    \| \zeta \|_{H^{-s} (\R)}^2 
    &:=  \int_\R (1 + \omega^2)^{-s} \big| \widehat \zeta (\omega) \big|^2 \, d\omega,
  \end{aligned}  
\end{equation}
where $s > 0$, $\widehat \zeta$ is the Fourier transform of $\zeta$ and $\mathcal S'(\R)$ is the space of tempered distributions, i.e., the dual of the Schwartz space $\mathcal S(\R)$. The reason for choosing this norm is that it is equivalent to an adjusted energy norm of $E$ for $s = (1-a)/2$, where we recall that $a$ is the strength of the singularity of $V$; see \eqref{Va}. In turn, this adjusted energy norm of $\tphi - \orho$ can be estimated in terms of energy differences, which are easier to analyse than $\tx$ itself. The benefit of using the fractional Sobolev norm over the adjusted energy norm is that it is independent of the interaction potential $V$. 

The main theorem of this paper is Theorem \ref{t}.

\begin{thm}[The quantitative estimate] \label{t}
Let $n \geq 1$. Let $E_n$ and $E$ be as defined in \eqref{En} and \eqref{E} with potentials $V$ and $U$ satisfying Assumption \ref{ass:VU} for some $0 \leq a < 1$. Then, the minimal values of $E_n$ and $E$ are attained, and the minimiser $\orho$ of $E$ is unique. Moreover, there exists $C > 0$ independent of $n$ such that for any minimiser $\tx$ of $E_n$,
$$ \| \tphi - \overline \rho \|_{H^{-(1-a)/2} (\R)}^2 
 \leq C \left\{ \begin{array}{ll}
  n^{-1+a} & 0 < a < 1 \\
  n^{-1} (\log n)^3 \: & a = 0,
  \end{array}  \right.$$
where $\tphi$ is constructed from $\tx$ by \eqref{phi:from:x}.
\end{thm}

The available tools for the proof of Theorem \ref{t} are the monotonicity and convexity of $V$ and the regularity properties of $\orho$ proven in \cite{KimuraVanMeurs19acc} (see Lemma \ref{l:orho} below). The difficulty is that no information on $\tx$ is available, except that $\tx$ is a minimiser of $E_n$. 
\smallskip

Next we give a sketch of the proof. For simplicity, we assume $a \in (0,1)$. The proof is divided in 3 steps. Step 1 is a preparatory step. In this step, we show that it is not restrictive to assume that 
\begin{subequations} \label{suppV:suppOrho}
\begin{gather} \label{suppV:cp}
  \supp V \text{ is compact,} \\
  \supp \orho = [0,1]. 
\end{gather} 
\end{subequations}

In Step 2, we estimate $\| \tphi - \overline \rho \|_{H^{-(1-a)/2} (\R)}^2$ without using any information on $\tx$ except that it is the minimiser of $E_n$. We start by following \cite{KimuraVanMeurs19acc} by rewriting $E$ as the sum of the square of a norm and a linear term, i.e., 
\begin{equation} \label{E:norm}
   E (\rho) = \frac12 \| \rho \|_V^2  + \int_\R U \, d\rho, \qquad 
   \| \rho \|_V^2 := \int_\R (V * \rho) \, d\rho.
\end{equation}
The norm $\| \cdot \|_V$ is the adjusted energy norm which we mentioned below \eqref{Hs}. The `adjustment' refers to the replacement of the interaction potential in Step 1 by a potential $V$ with compact support. We recall from \cite[Prop.\ 3.3]{KimuraVanMeurs19acc} that $\| \cdot \|_V$ is equivalent to the norm on $H^{-(1-a)/2} (\R)$. From \eqref{E:norm} and the minimality of $\orho$ (see Lemma \ref{l:orho}.\ref{l:orho:EL}), 
we obtain
\begin{align} \notag
  \| \tphi - \overline \rho \|_{H^{-(1-a)/2} (\R)}^2
  &\lesssim \| \tphi - \overline \rho \|_V^2 \\\notag
  &= 2 ( \overline \rho , \orho - \tphi )_V + \| \tphi \|_V^2 - \| \overline \rho \|_V^2 \\\notag
  &= 2 \int (V * \orho + U) \, d(\orho - \tphi) + 2 E (\tphi) - 2 E (\orho) \\\label{for:conv:rate:DtC:first:IDy:est}
  &\leq 0+ 2 ( E(\tphi) - E(\orho) ).
\end{align}
Then, proving Theorem \ref{t} translates into bounding the energy difference in the right-hand side.

To bound this difference, we obtain from the minimality of $\tx$ that
\begin{equation} \label{T1T2}
  E(\tphi) - E(\orho)
  = \underbrace{ E(\tphi) - E_n(\tx) }_{\Ta} 
    + \underbrace{ E_n(\tx) - E_n(\ox) }_{\leq 0}
    + \underbrace{ E_n(\ox) - E(\orho) }_{\Tb},
\end{equation}
where $\ox \in \Omega$ can be chosen freely. For an appropriate choice of $\ox$ (see \eqref{ox}), we bound $\Ta$ and $\Tb$ from above in Section \ref{s:Part2}. This is easy for the confinement term of the energy. For the interaction term, we perform a direct computation in which we write $\| \cdot \|_V^2$ explicitly as the integral over the square $(\supp \orho)^2 = [0,1]^2$, which we subdivide into the rectangles $(x_{i-1}, x_i) \times (x_{j-1}, x_j)$; see Figure \ref{fig:integration:domain:2D}. Then, by the monotonicity and convexity of $V$, we ultimately obtain 
\begin{equation} \label{T1T2:est}
  \Ta \leq \frac Cn + C' \NN (\tx)
  \quad \text{and} \quad
  \Tb \leq C n^{-1+a},
\end{equation}
where
\begin{equation} \label{NN}
     \NN (\bx) := \frac1{n^2} \sum_{i=1}^n V (x_i - x_{i-1})
\end{equation}   
is the part of the interaction term given by all nearest neighbour (superscript `nn') interactions. 

Putting the findings of Step 2 together, we obtain
\begin{equation*}
  \| \tphi - \overline \rho \|_{H^{-(1-a)/2} (\R)}^2
  \leq C n^{-1+a} + C' \NN (\tx).
\end{equation*}
Hence, it is left to show that 
\begin{equation} \label{NN:est} 
  \NN(\tx) \leq C n^{-1+a}.
\end{equation}

In Step 3 we prove \eqref{NN:est}. This is the difficult part of the proof of Theorem \ref{t}; we consider it as the main mathematical novelty of this paper. Our strategy is to establish the following lower bound on $E_n$:
\begin{equation} \label{En:LB}
  E_n(\bx) - E(\orho) \geq \NN(\bx) - C n^{-1+a} \quad \text{for all } \bx \in \Omega;   
\end{equation}
see Proposition \ref{prop:LB:En}. Then, taking $\bx = \tx$, the right-hand side in \eqref{En:LB} is bounded from above by $\Tb$. By the bound on $\Tb$ in \eqref{T1T2:est}, we then conclude \eqref{NN:est}. 

The argument in Step 3 is inspired by \cite[Sec.\ 2]{PetracheSerfaty14}; we also construct a renormalisation of the norm $\| \cdot \|_V$, but we need to construct a different one to allow for $\Vreg \neq 0$ and unbounded $\orho$.

We conclude the sketch of the proof by remarking that the case $a = 0$ can be treated analogously; the only difference is that the factor $\log n$ appears at a few places in the estimates. 
\smallskip

\subsection{Remarks on Theorem \ref{t}}
\label{s:intro:rems}

Here we list several remarks on the statement and the sketch of the proof of Theorem \ref{t}:

\paragraph{Uniform bound on support of $\tphi$} The proof of Step 1 contains the additional result that $\supp \tphi$ is bounded uniformly in $n$; see Proposition \ref{prop:VRn}. While the proof consists of common arguments in potential theory, we believe that the statement of Proposition \ref{prop:VRn} has merit on its own due to the rather weak assumptions on $V$ and $U$.

\paragraph{Extension of \cite{KimuraVanMeurs19acc}} The statement of Step 1 has further merit; the main results in \cite{KimuraVanMeurs19acc} (on the regularity of $\orho$) are stated under the stronger assumption that the tails of $V$ are integrable, in which case $V \in L^1(\R)$. Here, in Step 1 we extend these results to the larger class of potentials $V$ specified by Assumption \ref{ass:VU}.

\paragraph{Sharpness of the exponent} To test the degree of sharpness of the exponent of $n$ in the estimate in Theorem \ref{t}, we perform in Section \ref{s:num} numerical computations to compute $\| \tphi - \overline \rho \|_{H^{-(1-a)/2} (\R)}^2$ for various values of $n$. Surprisingly, our findings in Section \ref{s:num} show that in all test cases the numerical values of $\| \tphi - \overline \rho \|_{H^{-(1-a)/2} (\R)}^2$ decrease roughly as a power law $n^{-p}$ with $p$ significantly \emph{larger} than the exponent $1-a$ in Theorem \ref{t}. Hence, the exponent $1-a$ in Theorem \ref{t} appears not to be sharp. We comment in Section \ref{s:intro:disc} on the loss of sharpness in the proof of Theorem \ref{t} in the special case of $\Vreg \equiv 0$ (Riesz gases), for which more precise estimates than that in Theorem \ref{t} are available.

\paragraph{Choice of distance/norm} As mentioned before, our choice of norm circumvents any analysis of fine properties of $\tx$, and requires instead in the proof of Theorem \ref{t} the lower bound on the energy $E_n$ in \eqref{En:LB}. Other choices, such as $L^p$-norms or the Wasserstein distance, appear to require certain properties of $\tx$, which makes it more difficult to establish a quantitative estimate in these norms.

Other than mathematical convenience, our choice of norm has a further merit; it provides a quantitative estimate for the particle interaction force on $\R$ induced by the particle densities $\tphi$ and $\orho$. These interaction forces are 
\begin{equation*}
   -(V * \tphi)' - U'
   \quad \text{and} \quad
   -(V * \orho)' - U'
\end{equation*} 
respectively. To obtain a quantitative estimate from Theorem \ref{t}, we first change $V$ to have property \eqref{suppV:cp}\footnote{Changing $V$ changes the interaction forces, but only on a domain which is a certain distance away from $\supp \orho$ and $\supp \tphi$.}. Then, by \cite[Lem.\ 3.1(iii)]{KimuraVanMeurs19acc} it holds that 
\begin{equation*}
  \exists \, C > 0 \ \forall \, \omega \in \R :  \widehat V (\omega) 
  \leq C(1 + \omega^2)^{ -\tfrac{1-a}2}.
\end{equation*}
Hence, writing $\nu := \tphi - \orho$ and $s := \frac{1-a}2$,
\begin{equation*}
  \| \nu \|_{H^{-s} (\R)}^2
  = \int_\R (1 + \omega^2)^{-s} \big| \widehat \nu (\omega) \big|^2 \, d\omega
  \geq \frac1{C^2} \int_\R (1 + \omega^2)^{s} \big| \widehat V (\omega) \, \widehat \nu (\omega) \big|^2 \, d\omega
  = \frac1{C^2} \| V * \nu \|_{H^{s} (\R)}^2.
\end{equation*}
Thus,
\begin{equation*}
  \big\| (V * \tphi)' - (V * \orho)' \big\|_{H^{-(1+a)/2} (\R)}^2
  \leq C \big\| V * (\tphi - \orho) \big\|_{H^{(1-a)/2} (\R)}^2
  \leq C' \| \tphi - \orho \|_{H^{-(1-a)/2} (\R)}^2,
\end{equation*}
for which Theorem \ref{t} gives an upper bound.

\paragraph{Other choice of $\varphi$} Another commonly used choice for $\varphi$ than that in \eqref{phi:from:x} is to construct it from a one-dimensional Voronoi tessellation of the points $x_i$, and to assign to each Voronoi cell a mass of $1/n$. While Voronoi cells easily extend to higher dimensions, even in one dimension they introduce two complications for proving the corresponding estimates in \eqref{T1T2:est}. First, an additional choice for the Voronoi cells for $x_0$ and $x_n$ has to be made. Second, the mass of $\orho$ on the Voronoi cells constructed from $\ox$ may not equal $1/n$, which induces further error terms.

\subsection{Position in the literature}
\label{s:intro:disc} 

Here we put Theorem \ref{t} in the context of the literature. In particular, we show how it applies to problems in plasticity and in numerical integration, and how it compares to recent advances on Riesz gases.

\paragraph{Plasticity} The paper series started by 
\cite{GeersPeerlingsPeletierScardia13,
Hall11}
and continued in 
\cite{GarroniVanMeursPeletierScardia16,
HallHudsonVanMeurs18,
vanMeurs18Proc,
vanMeurs18,
VanMeursMunteanPeletier14}
studies the connection between models for plasticity of metals and an underlying microscopic model in a one-dimensional setting. This microscopic model is a minimisation problem of a certain $E_n$ of the form \eqref{En}. In particular, the interaction potential is
\begin{equation*}
  V(x) = x \coth x - \log (2 |\sinh x|),
\end{equation*}
which fits to Assumption \ref{ass:VU} with $a = 0$ and $\Vreg \not \equiv 0$. While in this paper series the convergence of $\tphi$ to $\orho$ as $n \to \infty$ is established, no quantitative estimates between $\tphi$ and $\orho$ were found\footnote{An exception is \cite{vanMeurs18Proc}, which establishes a quantitative estimate for a special, $n$-dependent choice of $U$. }, which limits the application to plasticity. 

The main result of this paper, Theorem \ref{t}, provides the first quantitative estimate for this microscopic model. The estimate is given by
\begin{equation} \label{appl:A1}
   \| \tphi - \orho \|_{H^{-1/2}}^2 
 \leq C(U) \frac{(\log n)^3}n.
\end{equation} 
For the application, a more detailed dependence of the constant $C$ on $U$ is needed. Since the proof of Theorem \ref{t} is constructive, it may be possible to use its steps for constructing an explicit expression for $C(U)$.

Our aim to establish quantitative estimates fits within a recent trend in mathematical research on plasticity. In this trend (see, e.g., \cite{HudsonVanMeursPeletier20ArXiv}), the starting point is a microscopic particle system which depends on several parameters such as $n$, the temperature and the lattice spacing between the atoms of the metal. The goal is to identify regions in the space of parameters on which the microscopic system can be approximated by a macroscopic system described in terms of a particle density $\rho$. To quantify this approximation, quantitative estimates such as that in Theorem \ref{t} are required. The main contribution from Theorem \ref{t} is that it requires no regularisation of the singularity of $V$ at $0$.

\paragraph{Approximation of functions} A common problem in the field of approximating a given function is how to choose finitely many sampling points at which to evaluate the function. For analytic functions defined on subsets of the complex plane, it is shown in \cite{TanakaSugihara19,
HayakawaTanaka19} that the optimal choice of sampling points can be obtained from the minimiser $\tx$ of $E_n$ for certain potentials $V$ and $U$ which satisfy Assumption \ref{ass:VU}. In particular, the interaction potential is explicitly given by $V(x) = -\log |\tanh x|$, which satisfies Assumption \ref{ass:VU} with $a = 0$ and $\Vreg \not \equiv 0$. Moreover, to bound the error made when replacing the given function by the approximation from the sampling points, it is required to find upper and lower bounds on
\begin{equation} \label{HT19}
  E_n(\tx) - E(\orho).
\end{equation}
The currently available bounds on \eqref{HT19} (see \cite[Thm.\ 2.3]{HayakawaTanaka19}) are comparable in size to $E_n(\tx)$ itself. Applying \eqref{En:LB} and the estimate on $T_2$ in \eqref{T1T2:est} yields
\begin{equation*} 
   \big| E_n(\tx) - E(\orho) \big| 
 \leq C(U) \frac{(\log n)^3}n,
\end{equation*} 
which demonstrates that it may be possible to construct a sharper estimate. In this setting, $U$ depends on $n$, and thus (similar to the application to plasticity) a more detailed estimate on $C(U)$ is required.

\paragraph{Riesz gases} For several other applications in approximation theory, detailed properties of the minimiser $\tx$ of $E_n$ are desired in the case where $\Vreg \equiv 0$ and where the particle positions may be of any dimension. Establishing such properties is the main topic of the paper series by Petrache, Sandier, Serfaty \emph{et al}.\ 
(\cite{SandierSerfaty152D,
SandierSerfaty15,
PetracheSerfaty14} 
to list a few). In particular, it is found in \cite[Thm.\ 4]{PetracheSerfaty14} that 
\begin{equation} \label{EnE:PS15}
   E_n(\tx) - E(\orho) = n^{-1+a}(-M + o(1))
   \qquad \text{as } n \to \infty
\end{equation} 
under appropriate restrictions on the potential $U$. The constant $M > 0$ 
 is of particular importance. It is explicitly expressed in terms of a maximisation problem on the microscopic configuration of $\tx$ in the limit $n \to \infty$. 

For the comparison between \eqref{EnE:PS15} and Theorem \ref{t}, it suffices to consider $M$ as a given constant. Even when \eqref{EnE:PS15} is restricted to one dimension, it is a more precise result than our estimates in \eqref{T1T2:est} and \eqref{En:LB}. The reason for obtaining such a precise result is that for $V = V_a$ the extension representation of \cite{CaffarelliSalsaSilvestre08} can be used (see \cite{PetracheSerfaty14} for details). For our larger class of potential $V$, we are not aware of a similar extension representation. Moreover, in this paper, $\orho$ need not be bounded, which complicates the estimates (see, e.g., Remark \ref{r:prop:LB:En} below).

Next, we put together the conclusion from \eqref{EnE:PS15} that $E_n(\tx) - E(\orho)$ decays as $n^{-1+a}$ and the finding from the numerical computations in Section \ref{s:num} that $\| \tphi - \overline \rho \|_V^2$ decays faster, namely as $n^{-p}$ with $p \geq -1+a$. First, we note that our numerical solutions for $\tx$ satisfy $x_i^* \in \supp \orho$ for all $i = 1,\ldots,n$ for each tested value of $n$. This implies that $\supp \tphi \subset \supp \orho$, and thus by Lemma \ref{l:orho}\ref{l:orho:EL} the second inequality in \eqref{for:conv:rate:DtC:first:IDy:est} turns into an equality. Substituting this equality in \eqref{T1T2}, we obtain
\begin{equation} \label{T1T2:Riesz} 
  \underbrace{ \| \tphi - \overline \rho \|_V^2 }_{O(n^{-p})}
  = \underbrace{ E(\tphi) - E_n(\tx) }_{\Ta} 
    + \underbrace{ E_n(\tx) - E(\orho) }_{n^{-1+a}(-M + o(1))}.
\end{equation}

From \eqref{T1T2:Riesz} it is clear that the loss of sharpness in the exponent of $n$ in the estimate in Theorem \ref{t} originates from the choice to estimate the energy differences in the right-hand side of \eqref{T1T2} without trying to optimise the multiplicative constant in front of $n^{-1+a}$. Finding this optimal multiplicative constant could result in a sharper estimate than that in Theorem \ref{t}. However, it seems quite difficult to optimise this constant. At least in the case of Riesz gases in \cite{PetracheSerfaty14} where \eqref{EnE:PS15} holds, a sharper estimate than that in Theorem \ref{t} follows from the right-hand side of \eqref{T1T2:Riesz} if one can show that
\begin{equation*}
  \Ta = n^{-1+a}(M + o(1)).
\end{equation*}
However, our current estimate $\Ta \leq C n^{-1+a}$ obtained from the strategy outlined in \eqref{T1T2:est} and \eqref{En:LB} is the core of our proof, and we have as of yet no clue on how the corresponding constant $C$ can be characterised as the constant $M$.

\paragraph{Organisation of the paper}  In Section \ref{s:not} we list our notation. In Sections \ref{s:Part1}, \ref{s:Part2} and \ref{s:Part3} we establish the required tools to justify the steps in the sketch of the proof of Theorem \ref{t}. In Section \ref{s:a0} we put these tools together to prove Theorem \ref{t}. In Section \ref{s:num} we describe and discuss our numerical findings for the actual dependence of the left-hand side in Theorem \ref{t} on $n$ and $a$.

\section{Notation}
\label{s:not}

The following table list the symbols which we use throughout the paper.

\newcommand{\specialcell}[2][c]{\begin{tabular}[#1]{@{}l@{}}#2\end{tabular}}
\begin{longtable}{lll}
$\wedge$, $\vee$ & $\alpha \wedge \beta := \min \{\alpha, \beta\}$ and $\alpha \vee \beta := \max \{\alpha, \beta\}$ & \\
$( \cdot, \cdot )_V$ & inner product constructed from $V$; $(f, g)_V = \int (V * g) f$ & \eqref{Vip}, \eqref{HV} \\
$\indicatornoacc A(x)$ & indicator function; $\indicatornoacc A(x) = 1$ if $x \in A$, and $\indicatornoacc A(x) = 0$ otherwise & \\
$a$ & strength of the singularity of $V$; $0 \leq a < 1$ & \eqref{Va} \\
$C, C', \ldots$ & some $n$-independent constants & \\
$E$ & energy for the particle density & \eqref{E} \\
$E_n$ & interacting particle energy; $E_n : \Omega \to \R$ & \eqref{En} \\
$\NN$ & nearest neighbour interactions; part of $E_n$ & \eqref{NN} \\
$\varphi$ & discrete density (piece-wise constant) constructed from $\bx \in \Omega$ & \eqref{phi:from:x} \\
$\Gamma$ & $\Gamma$-function; $\Gamma(\alpha) := \int_0^\infty x^{\alpha-1} e^{-x} \, dx$ & \\
$H^{-s} (\R)$ & fractional Sobolev space for $s > 0$ & \eqref{Hs} \\
$\ell_i$ & distance between nearest neighbours in $\bx \in \Omega$; $\ell_i := x_i - x_{i-1}$ & \eqref{m:ell} \\
$m_i$ & midpoints of nearest neighbours in $\bx \in \Omega$; $m_i := \frac12 (x_i + x_{i-1})$ & \eqref{m:ell} \\
$n$ & $n+1$ is the number of particles; $n \geq 1$ & \\
$\Omega$ & space of admissible particle configurations; $\Omega \subset \R^{n+1}$ & \eqref{Omega} \\
$\mathcal P (\R)$ & space of probability measures on $\R$ & \\
$\orho$ & the minimiser of $E$ & Lem.\ \ref{l:orho} \\
$U$ & confining potential & Ass.\ \ref{ass:VU} \\
$V$ & interaction potential & Ass.\ \ref{ass:VU} \\
$V_a$ & singular, homogeneous part of $V$ & \eqref{Va} \\
$\Vreg$ & regular part of $V$; $\Vreg = V - V_a$ & \\
$\tx$ & a minimiser of $E_n$; $\tx \in \Omega$ & \\
$\ox$ & particle configuration constructed from $\orho$; $\ox \in \Omega$ & \eqref{ox} \\
\end{longtable}

We use the convention that constants denoted by $C$ are independent of $n$ and may change from line to line. In several cases where the estimates are easier to follow when the change in constants is highlighted, we use $C', C'', \dots$ instead.

\section{Changing the tails of $V$}
\label{s:Part1} 

In this section we prove the key Propositions \ref{prop:VR} and \ref{prop:VRn} which will allow us to justify rigorously Step 1 of the sketch of the proof of Theorem \ref{t}, i.e., that $V$ may be replaced by a different interaction potential which satisfies \eqref{suppV:cp}. These propositions state that the set of minimisers of $E_n$ and the set of minimisers of $E$ do not depend on the tails of $V$. Since we consider these propositions of independent interest, we pose them under weaker conditions of $V$ than Assumption \ref{ass:VU}. In addition to proving these propositions, we show that $E_n$ and $E$ attain their minimal values, that the minimiser $\orho$ of $E$ is unique, and that $\orho$ satisfies  several properties.

\paragraph{Minimisers are independent of the tails of $V$} Let $V \in L_\loc^1(\R)$ satisfy \eqref{V:props} and $U$ satisfy Assumption \ref{ass:VU}. By an affine change of variables, we may assume that $[-1,1] \subset D(U)$. By \eqref{U:ass}, there exists a constant $M > 0$ such that 
\begin{equation} \label{U:LB}
  U(x) \geq \frac{|x|}M - M
  \quad \text{for all } x \in \R.
\end{equation}
By \eqref{V:props}, we note that, on $(0,\infty)$, $V$ is non-increasing and $V'$ is non-decreasing. Furthermore, $V'(x) \to 0$ as $x \to \infty$. Hence, there exists a point of differentiability $R \geq 2$ of $V$ for which
\begin{equation} \label{V:R}
  V'(R) \geq - \frac1{4M}
  \quad \text{and} \quad 
  \frac{V(R)}R \geq - \frac1{4M}.
\end{equation}

\begin{prop} \label{prop:VR}
Let $V \in L_\loc^1(\R)$ satisfy \eqref{V:props} and $U$ satisfy Assumption \ref{ass:VU}. Take $M, R > 0$ as in \eqref{U:LB} and \eqref{V:R}. Then, there exists a constant $S > 0$ independent of $V |_{(R, \infty)}$ such that any minimiser $\orho$ of $E$ satisfies $\supp \orho \subset [-S, S]$.
\end{prop}

\begin{proof}
We start by proving two auxiliary estimates. The first one is given by
\begin{equation} \label{V:N}
  \inf_{x > 0} \Big( V(x) + \frac x{4M} \Big)
  \geq \min_{0 < x \leq R} \Big( V(x) + \frac x{4M} \Big) \wedge 0 =: -N.
\end{equation}
To prove it, let $x > R$. Since $V$ is convex, the tangent line of $V$ at $R$ is below the graph of $V$. Then, by \eqref{V:R}, we obtain
\begin{align*}
  V(x) + \frac x{4M}
  \geq V(R) + (x - R)V'(R) + \frac x{4M}
  \geq -\frac R{4M} - \frac{x - R}{4M} + \frac x{4M}
  = 0.
\end{align*}
This proves \eqref{V:N}. We note that the constant $N \geq 0$ does not depend on $V |_{(R, \infty)}$. We set 
\begin{equation*}
  S := 2M \big(2 E ( \rho_0 ) + 2 + 2M + N \big),
\end{equation*}
where $\rho_0 := \frac12 \indicatornoacc{[-1,1]}$ is chosen rather arbitrarily to obtain that the value of $E ( \rho_0 )$ is finite and independent of $V |_{(R, \infty)}$.

The second auxiliary estimate is given by
\begin{equation} \label{VUU:LB}
  \frac{ V(x-y) + U(x) + U(y) }2
  \geq E ( \rho_0 ) + 1
  \quad \text{for all } x, y \text{ such that } |x| \vee |y| \geq S.
\end{equation}
To prove it, we first note from \eqref{U:LB} that
\begin{equation*}
  \frac{U(x) + U(y)}2
  \geq \frac{|x| + |y|}{2M} - M
  \geq \frac S{2M} - M 
  = 2 E ( \rho_0 ) + 2 + M + N.
\end{equation*}
Then, using \eqref{V:N},
\begin{equation} \label{VUU:2}
  V(x-y) + \frac{U(x) + U(y)}2
  \geq V(x-y) + \frac{|x-y|}{2M} - M
  \geq - N - M.
\end{equation}
Adding the above two estimates, we obtain \eqref{VUU:LB}.

Next, we prove Proposition \ref{prop:VR}. Suppose that $\orho$ is a minimiser of $E$ such that $\supp \orho \not \subset [-S, S]$. Then, $\overline m := \orho ([-S, S] ) < 1$. First, we claim that $\overline m > 0$. Indeed, if not, then by \eqref{VUU:LB}
\begin{align*}
  E(\orho)
  = \int_{\R} \int_{\R} \frac{ V(x-y) + U(x) + U(y) }2 \, d\orho(y) \, d\orho(x)
  \geq E ( \rho_0 ) + 1,
\end{align*}
which contradicts with the minimality of $\orho$. Using that $\overline m > 0$, we set $\rho := \orho |_{[-S, S]} / \overline m \in \cP(\R)$, and rely on \eqref{VUU:LB} to estimate
\begin{align*}
  E(\orho)
  &= \iint_{\R^2 \setminus [-S,S]^2} \frac{ V(x-y) + U(x) + U(y) }2 \, d\orho(y) \, d\orho(x) \\
  &\qquad  + \iint_{[-S,S]^2} \frac{ V(x-y) + U(x) + U(y) }2 \, \overline m^2 \, d\rho(y) \, d\rho(x) \\
  &\geq \big( E ( \rho_0 ) + 1 \big) (1 - \overline m^2) + \overline m^2 E(\rho)
  \geq \big( E ( \orho ) + 1 \big) (1 - \overline m^2) + \overline m^2 E(\rho).
\end{align*}
Rearranging terms,
\begin{equation*}
  E(\orho) \geq E(\rho) + \frac{1 - \overline m^2}{\overline m^2 } > E(\rho),
\end{equation*}
which contradicts with the minimality of $\orho$.
\end{proof} 

\begin{prop} \label{prop:VRn}
Let $V, U, M, R$ be as in Proposition \ref{prop:VR}, and let $n \geq 1$. Then, there exists a constant $S > 0$ independent of $n$ and of $V |_{(R, \infty)}$ such that any minimiser $\tx \in \Omega$ of $E_n$ satisfies $x_i^* \in [-S, S]$ for all $i = 0, \dots, n$.
\end{prop}

\begin{proof}
The proof is a discrete version of the proof of Proposition \ref{prop:VR}. We rely again on \eqref{V:N} with the same constant $N \geq 0$. Then, we set
\begin{equation} \label{S:n}
  S := \Big( 2N + 3 M + \int_{-1}^1 (V+U)(x) \, dx \Big) M.
\end{equation}

Since $\tx$ is ordered, it is enough to show that $-S \leq x_0^*$ and $x_n^* \leq S$. Suppose that $x_0^* < - S$ or $x_n^* > S$. We first treat the case in which both $x_0^* < - S$ and $x_n^* > S$ hold, and comment on the remaining case afterwards. We may assume that $U(x_0^*) \leq U(x_n^*)$, because otherwise we can obtain this by applying the variable transformation $x \mapsto -x$.

We will reach a contradiction with the minimality of $\tx$ by finding a lower energy state in $\Omega$. We construct this state by replacing $x_n^*$ by a more energetically favourable position $y_n \in [-1,1]$. With this aim, we first compute for any $y \in [-1,1]$
\begin{equation} \label{En:En:diff}
  n \big( E_n(\tx) - E_n(x_0^*, \dots, x_{n-1}^*, y) \big)
  = \underbrace{ \frac1n \sum_{j=0}^{n-1} V(x_n^* - x_j^*) + U(x_n^*) }_{W_n(x_n^*)}
    - \bigg( \underbrace{ \frac1n \sum_{j=0}^{n-1} V(y - x_j^*) + U(y) }_{W_n(y)} \bigg),
\end{equation}
where, for ease of notation, we have dropped the convention to have the particles positions ordered in the argument of $E_n$.

Since $W_n$ is lower semi-continuous and bounded from below on compact sets, it attains its minimum on $[-1,1]$. We take $y_n$ as a minimiser of $W_n$ over $[-1,1]$. Then, we estimate
\begin{equation*}
  2 W_n(y_n)
  \leq \int_{-1}^1 W_n(y) \, dy
  = \frac1n \sum_{j=0}^{n-1} \int_{-1}^1 V(y - x_j^*) \, dy + \int_{-1}^1 U(y)\, dy
  \leq \int_{-1}^1 (V+U)(y) \, dy.
\end{equation*}

Next we estimate $W_n(x_n^*)$ from below. Using that $V$ is non-increasing on $(0,\infty)$, we obtain
\begin{equation} \label{Wn:LB}
  W_n(x_n^*)
  \geq V(x_n^* - x_0^*) 
       + \frac{ U(x_n^*) + U(x_0^*) }4 
       + \frac{ U(x_n^*) - U(x_0^*) }4
       + \frac12 U(x_n^*).
\end{equation}
Then, following the estimates in \eqref{VUU:2} for the first two terms, and applying \eqref{U:LB} to the fourth term, we obtain
\begin{equation} \label{Wn:LB:2}
  W_n(x_n^*)
  \geq -(N + M) - 0 + \frac12 \Big( \frac{x_n^*}M - M \Big)
  > \frac S{2M} -N - \frac32 M.
\end{equation}

Collecting our results and substituting them in \eqref{En:En:diff} yields
\begin{equation*}
  n \big( E_n(\tx) - E_n(x_0^*, \dots, x_{n-1}^*, y_n) \big)
  > \frac S{2M} -N - \frac32 M - \frac12 \int_{-1}^1 (V+U)(y) \, dy,
\end{equation*}
which is non-negative by the choice of $S$ in \eqref{S:n}. This contradicts with the minimality of $\tx$.

Finally, we treat the case in which either $x_0^* < - S$ or $x_n^* > S$, but not both. Again, by changing variables if needed, we may assume that $x_n^* > S$, and thus $-x_0^* \leq S < x_n^*$. Then, the same proof can be adopted with a minor modification. This modification is to replace \eqref{Wn:LB} with the following:
\begin{equation*}
  W_n(x_n^*)
  \geq V(2 x_n^*) 
       + \frac{ U(x_n^*) + U(x_n^*) }4 
       + \frac12 U(x_n^*).
\end{equation*}
This results again in \eqref{Wn:LB:2}.
\end{proof}
 
\begin{rem}[$E_n$ attains its minimum] \label{r:En:minz:ex}
The existence of minimisers for $E_n$ is included in the proofs of Propositions \ref{prop:VR} and \ref{prop:VRn}. Indeed, since $E_n$ is continuous on $\Omega$ and $E_n (\bx) \to \infty$ as $\dist (\bx, \partial \Omega) \to 0$, it remains to be shown that $E_n (\bx) \to \infty$ as $|\bx| \to \infty$. To show this, we write
\begin{equation*}
  E_n (\bx) 
  = \frac1{2n^2} \sum_{i=0}^n \sum_{ j \neq i } \Big( V (x_i - x_j) + \frac{ \Vext(x_i) + \Vext(x_j) }2 \Big) + \frac1{2n} \sum_{i=0}^{n} \Vext (x_i)
\end{equation*}
and obtain from \eqref{U:LB} and \eqref{VUU:2} that 
\begin{equation*}
  E_n (\bx) 
  \geq - C + \frac{1}{2Mn} \sum_{i=0}^n |x_i|
  \xto{ |\bx| \to \infty } \infty.
\end{equation*}
\end{rem}

\paragraph{Properties of the compactly supported $V$ and $\orho$}
Let $V$ satisfy Assumption \ref{ass:VU} and the compact support condition \eqref{suppV:cp}. By Assumption \ref{ass:VU}, there exist constants $b,c > 0$ such that
\begin{equation} \label{Vpp:LB}
  V''(r) \geq c r^{-(2+a)} \qquad \text{for all } 0 < r \leq b.
\end{equation}
Since $\supp V$ is bounded, it is obvious from the convexity of $V$ on $\R \setminus \{0\}$ that $V \geq 0$. Moreover, for any $f,g \in L^2(\R)$,
\begin{equation} \label{Vip}
  \int_\R (V*f) g =: (f,g)_V
\end{equation}
defines an inner product, which induces the Hilbert space 
\begin{equation} \label{HV}
  \overline{ L^2(\R) }^{\| \cdot \|_V} \cong H^{-(1-a)/2} (\R).
\end{equation}
The proof of this is given in \cite[Lem.\ 3.1(iii) and Prop.\ 3.3]{KimuraVanMeurs19acc}, and uses that $V$ satisfies Assumption \ref{ass:VU} and $V \in L^1(\R)$. The following lemma is a simplified version of \cite[Thms.\ 1.4 and 1.5]{KimuraVanMeurs19acc}.

\begin{lem}[Properties of $\orho$] Let $0 \leq a < 1$. $E$ attains its minimal value on $\cP(\R)$. Its minimiser is unique, and has a density $\orho \in L^1(\R)$. Moreover, after applying an appropriate affine change of variables, $\orho$ satisfies \label{l:orho}
\begin{enumerate}[label=(\roman*)] 
  \item \label{l:orho:supp} $\supp \orho = [0,1]$;
  \item \label{l:orho:reg} $\orho \in L^1(0,1) \cap C((0,1))$;
  \item \label{l:orho:UB} $\exists \, C > 0 \ \forall \, 0 < x < 1 : \orho (x)  \leq C [x(1-x)]^{- \tfrac{1-a}2}$;
  \item \label{l:orho:EL} $
         \left\{ \begin{array}{ll}
           \displaystyle V * \orho + U - \overline C = 0
           & \text{on } [0,1] \\
           \displaystyle V * \orho + U - \overline C \geq 0
           & \text{on } \R.
         \end{array} \right\}$, where $\displaystyle \overline C := \int_\R (V * \orho) \, d\orho + \int_\R U \, d\orho > 0$.
\end{enumerate}
\end{lem}

The purpose of changing variables in Lemma \ref{l:orho} is to have $\supp \orho = [0,1]$ instead of some other bounded, closed interval. It is easy to see that Assumption \ref{ass:VU} and \eqref{suppV:cp} are invariant under an affine change of variables.

\section{Estimates on the terms $\Ta$ and $\Tb$}
\label{s:Part2}

In this section we fill in the details of Step 2 of the sketch of the proof of Theorem \ref{t}. We rely on the preparations in Step 1 to assume \eqref{suppV:suppOrho}, i.e., that the supports of $V$ and $\orho$ are compact.

Most of the estimates in the sketch are already rigorously justified. The only estimates left to prove are the two inequalities in \eqref{T1T2:est}, which we recall to be
\begin{equation} \label{TaTb:est:S6}
  \underbrace{ E(\tphi) - E_n(\tx) }_{\Ta} \leq \frac Cn + C' \NN (\tx)
  \quad \text{and} \quad
  \underbrace{ E_n(\ox) - E(\orho) }_{\Tb} \leq C n^{-1+a}.
\end{equation}
We note that $\ox \in \Omega$ is yet to be constructed, $C, C' > 0$ are independent of $n$, and $\NN$ is defined in \eqref{NN}. Recalling that $\supp \orho = [0,1]$, we define $\ox$ by
\begin{equation} \label{ox}
  \overline x_0 = 0, \quad
  \overline x_n = 1,
  \quad \text{and} \quad  
  \int_{\oox_{i-1}}^{\oox_i} \orho (x) \, dx = \frac1n 
  \quad \text{for all } i = 1,\dots,n.
\end{equation}

The structure of the proof of \eqref{TaTb:est:S6} is as follows. First, we split
\begin{equation*}
  \Tb = \underbrace{ E_n(\ox) - E(\ophi) }_{\Tc} + \underbrace{ E(\ophi) - E(\orho) }_{\Td},
\end{equation*}
where $\ophi$ is defined from $\ox$ by \eqref{phi:from:x}. We prove that $\Td \lesssim n^{-1+a}$ in Lemmas \ref{l:ophi:orho:U} and \ref{l:ophi:orho:V}. For $\Ta$ and $\Tc$, we note that -- except for the sign -- they are both of the form
\begin{equation*}
  E ( \varphi ) - E_n (\bx),
\end{equation*}
where $\bx$ and $\varphi$ are related through \eqref{phi:from:x}. In Lemma \ref{lem:bd:ene:diff:DtC} we show that
\begin{equation*}
  - \frac Cn -\NN (\bx) 
  \leq E ( \varphi ) - E_n (\bx)
  \leq \frac{C'}n + C'' \NN (\bx),
\end{equation*}
which yields
\begin{equation*}
  \Ta \leq \frac Cn + C' \NN (\tx) 
  \quad \text{and} \quad
  \Tc \leq \frac Cn + \NN (\ox).
\end{equation*}
Finally, in Lemma \ref{l:ox}\ref{l:ox:NN} we show that $\NN (\ox) \leq C n^{-1+a}$, which completes the proof of \eqref{TaTb:est:S6}. 

\paragraph{Properties of $\ox$ and the bound on $\NN (\ox)$}

We start by introducing some notation. First, for given $\bx \in \Omega$, we define 
\begin{equation} \label{m:ell}
  m_i := (x_i + x_{i-1})/2,
  \quad \text{and} \quad
  \ell_i := x_i - x_{i-1}
\end{equation}
where $m := (m_1, \ldots, m_n) \in \R^n$ lists the midpoints of neighbouring particles, and $\ell := (\ell_1, \ldots, \ell_n) \in \R^n$ the distances between them. These quantities are illustrated in Figure \ref{fig:integration:domain:2D}. For the specific choices $\ox$ and $\tx$, we denote the related midpoints and interparticle distances as $\ovm_i$, $\oell_i$ and $\tm_i$, $\tell_i$ respectively.

\begin{figure}[h]
\centering
\begin{tikzpicture}[scale=6]
    \def \xo {0}
    \def \xi {0.2}
    \def \xii {0.55}
    \def \xiii {1}
    \def \yi {\xi/2 + \xo/2}
    \def \yii {\xii/2 + \xi/2}
    \def \yiii {\xiii/2 + \xii/2}
    \def \r {0.02}

    \draw[->] (0,0) -- (1.2, 0);
    \draw[->] (0,0) -- (0, 1.2);
    \draw[dashed] (0,0) -- (\xiii, \xiii);
    
    \draw (\xo, -\r) node[below] {$x_0$};
    \draw (\xi, -\r) node[below] {$x_1$};
    \draw (\xii, -\r) node[below] {$x_2$};
    \draw (\xiii, -\r) node[below] {$x_3$};
	\draw (-\r, \xo) node[left] {$x_0$};
    \draw (-\r, \xi) node[left] {$x_1$};
    \draw (-\r, \xii) node[left] {$x_2$};
    \draw (-\r, \xiii) node[left] {$x_3$};  
    \draw (\yi, -\r) node[below] {$m_1$} -- (\yi, \r);
    \draw (\yii, -\r) node[below] {$m_2$} -- (\yii, \r);
    \draw (\yiii, -\r) node[below] {$m_3$} -- (\yiii, \r);
    \draw (-\r, \yi) node[left] {$m_1$} -- (\r, \yi);
    \draw (-\r, \yii) node[left] {$m_2$} -- (\r, \yii);
    \draw (-\r, \yiii) node[left] {$m_3$} -- (\r, \yiii);   
    
    \foreach \x in {\xo, \xi, \xii, \xiii} {
      \foreach \y in {\xo, \xi, \xii, \xiii} {
        \draw[fill = black, black] (\x, \y) circle (\r); %
      }
    }
    
    \foreach \x in {\yi, \yii, \yiii} {
      \foreach \y in {\yi, \yii, \yiii} {
        \draw (\x, \y) circle (\r); %
      }
    }
    
    \draw (\xi, \xii) rectangle (\xii, \xiii);
    \draw (\yii, \xiii) node[above] {$\ell_2$};
    \draw (\xii, \yiii) node[right] {$\ell_3$};
\end{tikzpicture}
\caption{The integration domain.}
\label{fig:integration:domain:2D}
\end{figure}
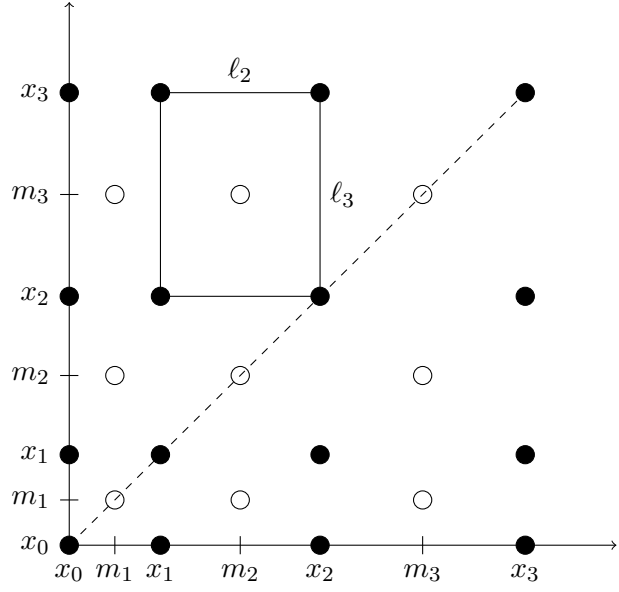

Second, we introduce
\begin{equation} \label{Dn}
  D_n : \Omega \to [0, \infty), \qquad
  D_n(\bx) := \frac12 \frac1{n^2} \sum_{i=1}^n \frac1{\ell_i^2} \iint_{( 0, \ell_i )^2} V (x - y) \, dy dx,
\end{equation}
where $\ell_i$ depends on $\bx$ through \eqref{m:ell}. We note that
\begin{equation} \label{NN:leq:Dn}
  \NN(\bx) \leq 2 D_n(\bx).
\end{equation} 

\begin{lem}[Properties of $\ox$] \label{l:ox}
There exist constants $C, c > 0$ such that for all $n \geq 1$
\begin{enumerate}[label=(\roman*)]
  \item \label{l:ox:xi} $\displaystyle \overline x_i \geq c \Big( \frac in \Big)^{\tfrac2{1+a}}
         \quad \text{and} \quad
         \overline x_{n-i} \leq 1 - c \Big( \frac in \Big)^{\tfrac2{1+a}}
         \quad \text{for all } i = 0,\ldots, \Big\lfloor \frac n2 \Big\rfloor$;
  \item \label{l:ox:elli} $\displaystyle \oell_1 \wedge \oell_n \geq c \Big( \frac 1n \Big)^{\tfrac2{1+a}}
         \quad \text{and} \quad
         \oell_i \wedge \oell_{n+1-i} \geq \frac cn \Big( \frac{i-1}n \Big)^{\tfrac{1-a}{1+a}}
         \quad \text{for all } i = 2,\ldots, \Big\lfloor \frac n2 \Big\rfloor$;
  \item \label{l:ox:NN} $D_n(\ox) + \NN (\ox) \leq C n^{-1+a}$.
\end{enumerate}

\end{lem}

\begin{proof}
For convenience we assume that $n$ is even. Since $\overline x_0 = 0$, it is sufficient to consider any $i \geq 1$. Using Lemma \ref{l:orho}\ref{l:orho:UB}, we find that
\begin{equation*}
  \frac in 
  = \int_0^{\overline x_i} \orho
  \leq \int_0^{\overline x_i} C x^{-\tfrac{1-a}2} \, dx 
  = C' \overline x_i^{\tfrac{1+a}2}.
\end{equation*}
This implies the first part of Property \ref{l:ox:xi}. The estimate for $\overline x_{n-i}$ is found analogously.

Next we bound $\oell_i$ from below. For $i = 1$, we find
$$
  \oell_1 
  = \overline x_1
  \geq c \Big( \frac 1n \Big)^{\tfrac2{1+a}}.
$$ 
For $i \geq 2$, we estimate similarly as above
\begin{equation*}
  \frac 1n 
  = \int_{\overline x_{i-1}}^{\overline x_i} \orho
  \leq \int_{\overline x_{i-1}}^{\overline x_i} C x^{-\tfrac{1-a}2} \, dx 
  = C' \Big( \overline x_i^{\tfrac{1+a}2} - \overline x_{i-1}^{\tfrac{1+a}2} \Big).
\end{equation*} 
Inserting $\overline x_i = \oell_i + \overline x_{i-1}$, we obtain
\begin{equation} \label{p:li:LB}
  \oell_i \geq \Big( \frac1{C' n} + \overline x_{i-1}^{\tfrac{1+a}2} \Big)^{\tfrac2{1+a}} - \overline x_{i-1}.
\end{equation}
Since $\frac2{1+a} > 1$, the function $\psi(t) := t^{2/(1+a)}$ is convex for $t > 0$, and thus $\psi(t + \varepsilon) \geq \psi(t) + \varepsilon \psi'(t)$ for all $t, \varepsilon > 0$. Applying this inequality to \eqref{p:li:LB}, and then using Property \ref{l:ox:xi}, we obtain
\begin{equation*}
  \oell_i 
  \geq \frac1{C' n} \frac2{1+a} \Big( \overline x_{i-1}^{\tfrac{1+a}2} \Big)^{\tfrac2{1+a} - 1}
  = \frac cn \overline x_{i-1}^{\tfrac{1-a}2} 
  \geq \frac{c'}n \Big( \frac{i-1}n \Big)^{\tfrac{1-a}{1+a}}.
\end{equation*}
The estimate for $\oell_{n+1-i}$ is found analogously.

Finally we prove Property \ref{l:ox:NN}. By \eqref{NN:leq:Dn} it is enough to estimate $D_n(\ox)$.
From $V(x) \leq C/|x|^a$ and Property \ref{l:ox:elli} we obtain
\begin{multline*}
D_n(\ox)
  = \frac12 \frac1{n^2} \sum_{i=1}^n \iint_{( 0, 1 )^2} V (\oell_i (x - y)) \, dy dx 
  \leq \frac C{n^2} \bigg( \oell_1^{-a} + \oell_n^{-a} + \sum_{ i = 2 }^{n-1} \overline \ell_i^{-a} \bigg) \\
  \leq C n^{-2 + \tfrac{2 a}{1+a}} + C \frac{n^a}{n} \frac1n \sum_{ i = 2 }^{n/2} \Big( \frac{i-1}n \Big)^{-a\tfrac{1-a}{1+a}}
  \leq C n^{-1+a}.
\end{multline*}
\end{proof}

\paragraph{The bound on $\Td = E(\ophi) - E(\orho)$}

We recall from \eqref{E:norm} that $E$ consists of an interaction part and a confinement part. For $E(\ophi) - E(\orho)$, we bound these terms separately in  Lemmas \ref{l:ophi:orho:V} and \ref{l:ophi:orho:U} respectively.

\begin{lem} \label{l:ophi:orho:U} For all $n \geq 1$
\begin{equation*}
  \int_0^1 U (x) \, (\ophi - \orho)(x) \, dx
  \leq \frac{U(0) + U(1)}n.
\end{equation*}
\end{lem}

\begin{proof}
From \eqref{phi:from:x} and \eqref{ox} we observe that the densities $\ophi$ and $\orho$ have mass $1/n$ on $[\overline x_{i-1}, \overline x_{i}]$ for each $i$. We use this to 
 estimate
\begin{align} \label{pf:U}
  \int_0^1 U (\ophi - \orho)
  = \sum_{i=1}^n \int_{\overline x_{i-1}}^{\overline x_i} U (\ophi - \orho)
  \leq \frac1n \sum_{i=1}^n \Big( \max_{[\overline x_{i-1}, \overline x_i]} U - \min_{[\overline x_{i-1}, \overline x_i]} U \Big).
\end{align}
Assume for convenience that $\overline x_I$ is a minimiser of $U$ for some $I \in \{1,\ldots,n\}$. Then, since $U$ is convex, it is non-increasing on $[0, \overline x_I]$, and thus
\[
  \frac1n \sum_{i=1}^I \Big( \max_{[\overline x_{i-1}, \overline x_i]} U - \min_{[\overline x_{i-1}, \overline x_i]} U \Big)
  = \frac1n \sum_{i=1}^I \big( U(\overline x_{i-1}) - U(\overline x_{i}) \big)
  = \frac1n \big( U(\overline x_0) - U(\overline x_I) \big)
  = \frac{U(0)}n.
\]
A similar argument on $[\overline x_I, 1]$ yields that the remaining terms of the sum in the right-hand side of \eqref{pf:U} equal $U(1)/n$, and the statement of Lemma \ref{l:ophi:orho:U} follows. In the case where the interval of minimisers of $U$ is contained in $(\overline x_{I-1}, \overline x_{I})$ for some $I$, a similar argument applies.
\end{proof}

\begin{lem} \label{l:ophi:orho:V} There exists a constant $C > 0$ such that for all $n \geq 1$
\begin{equation*}
  \| \overline \varphi  \|_V^2 - \| \overline \rho \|_V^2 \leq C n^{-1+a}.
\end{equation*}
\end{lem}

\begin{proof}
We write
\begin{equation*}
  \| \overline \varphi  \|_V^2 - \| \overline \rho \|_V^2
  = \sum_{i=1}^n \sum_{ j=1 }^n \int_{\overline x_{i-1}}^{\overline x_i} \int_{\overline x_{j-1}}^{\overline x_j} V (x - y) \, \big( \ophi(x) \ophi(y) - \orho(x) \orho(y) \big) \, dy dx 
  =: \Te + \Tf + \Tg, 
\end{equation*}
where the terms $\Te$, $ \Tf$ and $\Tg$ correspond to the part of the sum where $i - j = 0$, $|i-j| = 1$ and $|i-j| \geq 2$ respectively. We bound all these three terms separately. With this aim, we set 
\begin{equation*}
  \overline \varphi_i := \overline \varphi \indicatornoacc{ ( \overline x_{i-1} , \overline x_i ) }
  \qquad \text{for } i =1, \ldots, n,
\end{equation*}
and note that, by \eqref{Dn} and Lemma \ref{l:ox}\ref{l:ox:NN}, 
\begin{equation*}
  \sum_{i=1}^n \| \ophi_i \|_V^2
  = 2 D_n(\ox)
  \leq C n^{-1+a}.
\end{equation*}

For $\Te$ we simply estimate  
\begin{align*}
  \Te &= \sum_{i=1}^n \int_{\overline x_{i-1}}^{\overline x_i} \int_{\overline x_{i-1}}^{\overline x_i} V (x - y) \, \big( \ophi(x) \ophi(y) - \orho(x) \orho(y) \big) \, dy dx \\
  &\leq \sum_{i=1}^n \int_{\overline x_{i-1}}^{\overline x_i} \int_{\overline x_{i-1}}^{\overline x_i} V (x - y) \ophi(x) \ophi(y) \, dy dx 
  = \sum_{i=1}^n \| \ophi_i \|_V^2
  \leq C n^{-1+a}.
\end{align*}
For $\Tf$, we similarly obtain
\begin{align} \notag 
  \Tf &= 2 \sum_{i=1}^{n-1} \int_{\overline x_{i}}^{\overline x_{i+1}} \int_{\overline x_{i-1}}^{\overline x_i} V (x - y) \, \big( \ophi(x) \ophi(y) - \orho(x) \orho(y) \big) \, dy dx \\\label{p:T1:UB}
  &\leq 2\sum_{i=1}^{n-1} (\ophi_{i+1}, \ophi_i)_V
  \leq \sum_{i=1}^{n-1} \big( \| \ophi_{i+1} \|_V^2 + \| \ophi_{i} \|_V^2 \big)
  \leq C n^{-1+a}.
\end{align} 

Finally we estimate $\Tg$. We note that in the integrals in the terms of $\Tg$, the singularity of $V$ is avoided. This allows for pointwise evaluation of the integrand. By using that $V$ is even, and non-increasing on the positive axis, we estimate
\begin{align*}
  \Tg
  &= 2 \sum_{i=3}^n \sum_{ j=1 }^{i-2} \bigg( \int_{\overline x_{i-1}}^{\overline x_i} \int_{\overline x_{j-1}}^{\overline x_j} V (x-y) \, \big( \ophi(x) \ophi(y) - \orho(x) \orho(y) \big) \, dy dx  \\ 
  &\leq 2 \sum_{i=3}^n \sum_{ j=1 }^{i-2} \bigg( \int_{\overline x_{i-1}}^{\overline x_i} \int_{\overline x_{j-1}}^{\overline x_j} V (\overline x_{i-1} - \overline x_{j}) \, \ophi(x) \ophi(y) \, dy dx  \\ 
  &  \qquad \qquad \qquad \qquad - \int_{\overline x_{i-1}}^{\overline x_i} \int_{\overline x_{j-1}}^{\overline x_j} V (\overline x_{i} - \overline x_{j-1}) \, \orho(x) \orho(y)  \, dy dx \bigg) \\
  &= \frac2{n^2} \sum_{i=3}^n \sum_{ j=1 }^{i-2} \big(
       V (\overline x_{i-1} - \overline x_j) - V (\overline x_i - \overline x_{j-1}) \big).
\end{align*}
We recognise a telescopic series after changing the summation index to $k = i+j-1$:
\begin{align*}
  \Tg
  &\leq \frac2{n^2} \sum_{i=3}^n \sum_{ j=1 }^{i-2} \big(
       V (\overline x_{i-1} - \overline x_j) - V (\overline x_i - \overline x_{j-1}) \big) \\
  &= \frac2{n^2} \sum_{i=3}^n \sum_{ k = i }^{2i - 3} \big(
       V (\overline x_{i-1} - \overline x_{k-(i-1)}) - V (\overline x_i - \overline x_{k-i}) \big) \\
       &= \frac2{n^2} \sum_{ k = 3 }^{2n - 3} \sum_{i=\lceil \frac{k+3}2 \rceil}^{k \wedge n} \big(
       V (\overline x_{i-1} - \overline x_{k-(i-1)}) - V (\overline x_i - \overline x_{k-i}) \big) \\
  &= \frac2{n^2} \sum_{ k = 3 }^{2n - 3} \big(
       V (\overline x_{\lceil \frac{k+1}2 \rceil} - \overline x_{\lfloor \frac{k-1}2 \rfloor}) - V (\overline x_{k \wedge n} - \overline x_{0 \vee (k-n)}) \big) \\
  &\leq \frac2{n^2} \sum_{ k = 3 }^{2n - 3}
       V (\overline x_{\lceil \frac{k+1}2 \rceil} - \overline x_{\lfloor \frac{k-1}2 \rfloor})
  \leq \frac{4}{n^2} \sum_{ i = 2 }^{n - 1}
       V (\overline \ell_i)
  \leq 4 \NN(\ox)
\end{align*}
which, by Lemma \ref{l:ox}\ref{l:ox:NN}, is bounded by $C n^{-1+a}$.
\end{proof}

\paragraph{The upper and lower bound on $E(\varphi) - E_n (\bx)$}

First, we state and prove the opposite inequality in \eqref{NN:leq:Dn} as an auxiliary result.

\begin{lem} \label{l:Dn:by:NN} There exists constants $C, C' > 0$ such that for all $n \geq 1$ and all $\bx \in \Omega$
$$
  D_n(\bx) \leq C \NN(\bx) + \frac{C'}n.
$$
\end{lem}

\begin{proof}
Using $V = V_a + \Vreg$, we split $D_n(\bx)$ in two parts:
\begin{equation*}
  D_n(\bx)
  = \frac12 \frac1{n^2} \sum_{i=1}^n \frac1{\ell_i^2} \iint_{( 0, \ell_i )^2} V_a (x - y) \, dy dx
    + \frac12 \frac1{n^2} \sum_{i=1}^n \frac1{\ell_i^2} \iint_{( 0, \ell_i )^2} \Vreg (x - y) \, dy dx.
\end{equation*}
The first part can be computed explicitly. This yields
\begin{equation*}
  \frac12 \frac1{n^2} \sum_{i=1}^n \frac1{\ell_i^2} \iint_{( 0, \ell_i )^2} V_a (x - y) \, dy dx
  = \frac{C_a}{n^2} \sum_{i=1}^n V_a (\ell_i)
\end{equation*}
for some explicit constant $C_a > 0$. For the second term, we rely on the regularity of $\Vreg$ to estimate
\begin{equation*}
  \frac12 \frac1{n^2} \sum_{i=1}^n \frac1{\ell_i^2} \iint_{( 0, \ell_i )^2} \Vreg (x - y) \, dy dx
  \leq \frac1{n} \Big( \frac12 + C_a \Big) \| \Vreg \|_{C([-1,1])} + \frac{C_a}{n^2} \sum_{i=1}^n \Vreg (\ell_i).
\end{equation*}
Collecting these findings, we obtain the estimate in Lemma \ref{l:Dn:by:NN}:
\[
  D_n(\bx)
  \leq \frac{C}{n^2} \sum_{i=1}^n (V_a + \Vreg) (\ell_i) + \frac{C'}n
  = C \NN (\bx) + \frac{C'}n.
\]
\end{proof}

\begin{lem} [Energy bounds on the piecewise constant approximation] \label{lem:bd:ene:diff:DtC} There exists $C \geq 0$ such that for all $n \geq 1$ and all $\bx \in \Omega$
\begin{equation*}
  -\NN (\bx) - \frac1n (U(0) + U(1)) 
  \leq E ( \varphi ) - E_n (\bx)
  \leq C \Big( \NN (\bx) + \frac1n \Big),
\end{equation*}
where $\varphi$ is the piece-wise constant function constructed from $\bx$ by \eqref{phi:from:x}. 

\end{lem}

\begin{proof} We divide the proof in four steps. In Step 1 we bound the confinement part of $E ( \varphi ) - E_n (\bx)$, and in Steps 2 -- 4 we bound the interaction part. Given $\bx$, we let $m_i$ and $\ell_i$ be defined by \eqref{m:ell} (see Figure \ref{fig:integration:domain:2D}).

\smallskip

\emph{Step 1: bounds on the confinement part.} The confinement part of $E ( \varphi ) - E_n (\bx)$ is given by
\begin{equation*}
  F_n(\bx) := \int_0^1 U \varphi - \frac1n \sum_{i=0}^n U(x_i)
  = \frac1n \sum_{i=1}^n \frac1{\ell_i} \int_{x_{i-1}}^{x_i} U - \frac1n \sum_{i=0}^n U(x_i).
\end{equation*}
Since $U \geq 0$ is convex, it is easy to see that $- \frac1n (U(0) + U(1)) \leq F_n(\bx) \leq 0$.

In the remainder of the proof, we set $U \equiv 0$ to focus on the interaction part.
\smallskip

\emph{Step 2: rewriting $E ( \varphi ) - E_n (\bx)$ as a sum of error terms.} We show that 
\begin{equation} \label{p:Ediff:DQCB}
  E ( \varphi ) - E_n (\bx)
  = D_n (\bx) + Q_n(\bx) - R_n (\bx) - B_n(\bx),
\end{equation}
where the four \emph{non-negative} error terms are given by \eqref{Dn} and
\begin{align*} 
  Q_n(\bx) \, &:= \frac1{n^2} \sum_{i=2}^n \sum_{ j=1 }^{i-1} \bigg[ \frac1{\ell_i \ell_j} \int_{x_{i-1}}^{x_i} \int_{x_{j-1}}^{x_j} V (  x - y ) \, dy dx - V( m_i - m_j ) \bigg],  \\
  R_n(\bx) &:= \frac1{n^2} \sum_{i=2}^n \sum_{ j=1 }^{i-1} \Big( \frac12 V( x_i - x_j ) + \frac12 V( x_{i - 1} - x_{j - 1} ) - V( m_i - m_j ) \Big), \\
  B_n(\bx) &:= \frac12 \frac1{n^2} \sum_{i=1}^{n} \big[ V (x_i - x_0) + V (x_{n} - x_{i-1}) \big].
\end{align*}
Indeed, \eqref{p:Ediff:DQCB} follows from
\begin{align} \notag
  E ( \varphi ) 
  &= \frac12 \int_0^1 \int_0^1 V (x - y) \varphi (x) \varphi (y) \, dy dx \\\label{Ephi:nice}
  &= \frac12 \sum_{i=1}^n \sum_{ j=1 }^n \int_{x_{i-1}}^{x_i} \int_{x_{j-1}}^{x_j} V (x - y) \frac{1/n}{ x_i - x_{i-1} } \frac{1/n}{ x_j - x_{j-1} } \, dy dx \\\notag
  &= \frac1{n^2} \sum_{i=2}^n \sum_{ j=1}^{i-1} \frac1{\ell_i \ell_j} \int_{x_{i-1}}^{x_i} \int_{x_{j-1}}^{x_j} V (  x - y ) \, dy dx + D_n (\bx) \\\notag
  &= \frac1{n^2} \sum_{i=2}^n \sum_{ j=1}^{i-1} V (m_i - m_j) + (D_n + Q_n) (\bx) \\\notag
  &= \frac1{2 n^2} \sum_{i=2}^n \sum_{ j=1}^{i-1} \big( V( x_i - x_j ) + V( x_{i - 1} - x_{j - 1} ) \big) + (D_n + Q_n - R_n) (\bx) \\\notag
  &= E_n(\bx) + (D_n + Q_n - R_n - B_n) (\bx).
\end{align}

\smallskip

\emph{Step 3: the lower bound for $E ( \varphi ) - E_n (\bx)$.}
Since the error terms $D_n$, $Q_n$, $R_n$ and $B_n$ are all non-negative, we observe from \eqref{p:Ediff:DQCB} that it is enough to show that
\begin{equation*}
  R_n(\bx) + B_n(\bx) \leq \NN (\bx).
\end{equation*}
By using $m_i - m_j \leq x_i - x_{j-1}$, we obtain this estimate from
\begin{align*}
  R_n(\bx) + B_n(\bx)
  &= \frac1{n^2} \sum_{i=1}^n \sum_{ j=0 }^{i-1} V(  x_i -  x_j ) - \frac1{n^2} \sum_{i=2}^n \sum_{ j=1 }^{i-1} V(  m_i -  m_j ) \\
  &\leq \frac1{n^2} \sum_{i=1}^n \sum_{ j=0 }^{i-1} V(  x_i -  x_j ) - \frac1{n^2} \sum_{i=2}^n \sum_{ j=0 }^{i-2} V(  x_i -  x_j )
  = \NN (\bx).
\end{align*}

\emph{Step 4: the upper bound for $E ( \varphi ) - E_n (\bx)$.} Since $B_n \geq 0$, it is enough to show that
\begin{equation*} 
  D_n(\bx) + Q_n (\bx) - R_n (\bx)
  \leq C \Big( \NN (\bx) + \frac1n \Big).
\end{equation*}
Then, by Lemma \ref{l:Dn:by:NN}, it suffices to show that $Q_n - R_n \leq 2 D_n$. Writing
\begin{multline*}
  Q_n(\bx) - R_n(\bx) \\
  = \frac1{n^2} \sum_{i=2}^n \sum_{ j=1 }^{i-1} \bigg[ \frac1{\ell_i \ell_j} \int_{x_{i-1}}^{x_i} \int_{x_{j-1}}^{x_j} V (  x - y ) \, dy dx - \Big( \frac12 V( x_i - x_j ) + \frac12 V( x_{i - 1} - x_{j - 1} ) \Big) \bigg],
\end{multline*}
we use convexity of $V$ to bound the integral for $i \geq j + 2$ by
\begin{equation*}
  \frac1{\ell_i \ell_j} \int_{x_{i-1}}^{x_i} \int_{x_{j-1}}^{x_j} V (  x - y ) \, dy dx
  \leq \frac12 V( x_{i-1} - x_j ) + \frac12 V( x_i - x_{j - 1} ).
\end{equation*}
This yields
\begin{multline} \label{p:QmC:UB}
  Q_n(\bx) - R_n(\bx) 
  \leq \frac1{n^2} \sum_{i=1}^{n-1} \frac1{\ell_{i+1} \ell_i} \int^{x_{i+1}}_{x_i} \int_{x_{i-1}}^{x_i} V (  x - y ) \, dy dx \\
      + \frac1{2 n^2} \bigg( \sum_{i=3}^n \sum_{ j=1 }^{i-2} [V( x_{i-1} - x_j ) + V( x_i - x_{j - 1} )]
  - \sum_{i=2}^n \sum_{ j=1 }^{i-1} [ V( x_i - x_j ) + V( x_{i - 1} - x_{j - 1} ) ] \bigg).
\end{multline}
For the term within parentheses, a change of index readily reveals that the second summation includes all terms of the first summation. We then use $V \geq 0$ to estimate this term from above by $0$. The remaining term in \eqref{p:QmC:UB} can be estimated similarly as in \eqref{p:T1:UB}. This yields $Q_n - R_n \leq 2 D_n$.
\end{proof}

\section{Lower bound on $E_n$}
\label{s:Part3}

In this section we prove \eqref{En:LB}, which is the crucial step in Step 3 of the sketch of the proof of Theorem \ref{t}. More precisely, we assume that $V$ and $\orho$ satisfy \eqref{suppV:suppOrho}, and prove the following proposition.

\begin{prop} [Lower bound on $E_n$] \label{prop:LB:En} 
There exists $C > 0$ such that for all $n \geq 1$ and all $\bx \in \Omega$
\begin{equation*} 
  E_n (\bx) - \NN (\bx) - E ( \orho ) 
  \geq -C n^{-1+a}.
\end{equation*}
\end{prop}

We give the proof of Proposition \ref{prop:LB:En} after a preliminary construction of a renormalised norm of $\| \cdot \|_V$. With this aim, we introduce
\begin{equation} \label{Vn}
  V_n (r) := \left\{ \begin{aligned}
    & V(\tfrac1n) + (r - \tfrac1n) V'(\tfrac1n)
    && \text{if } 0 \leq r \leq \tfrac1n \\
    & V(r)
    && \text{if } r > \tfrac1n
  \end{aligned} \right.
\end{equation}
with even extension to the negative half-line. Figure \ref{fig:V:puntmuts} illustrates a typical example of $V$ and $V_n$. Lemma \ref{l:Vn} lists several basic properties of $V_n$.

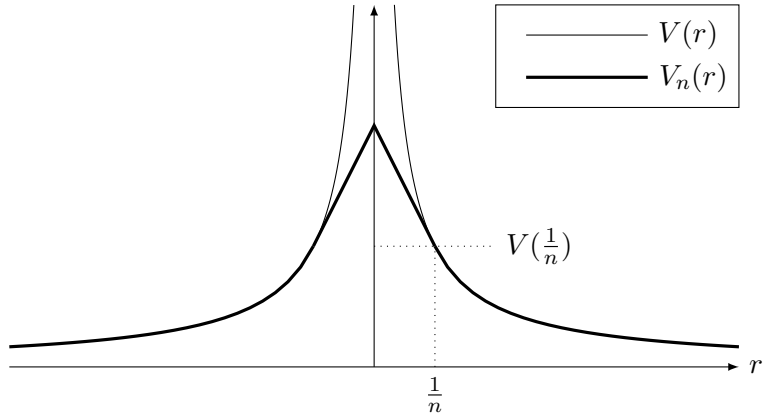
\begin{figure}[ht]
\centering
\begin{tikzpicture}[scale=0.8, >= latex]
\draw[dotted] (1,0) node[below] {$\tfrac1n$} -- (1,2);
\draw[dotted] (0,2) --++ (2,0) node[right] {$V(\tfrac1n)$};             
\draw[->] (-6,0) -- (6,0) node[right] {$r$};
\draw[->] (0,0) -- (0,6);
\draw[very thick] (-1,2) -- (0,4) -- (1,2);
\draw[very thick, domain=-6:-1] plot (\x,{-2/\x});
\draw[domain=-1:-0.3333] plot (\x,{-2/\x});
\draw[domain=0.3333:1] plot (\x,{2/\x});
\draw[very thick, domain=1:6] plot (\x,{2/\x});
\draw (2,4.3) rectangle (6,6);
\draw (2.5,5.5) -- (4.5,5.5) node[right] {$V(r)$};
\draw[very thick] (2.5,4.8) -- (4.5,4.8) node[right] {$V_n(r)$};
\end{tikzpicture} \\
\caption{The piecewise-affine regularisation $V_n$ of the interaction potential $V$.}
\label{fig:V:puntmuts}
\end{figure}

\begin{lem} [Properties of $V_n$] \label{l:Vn} 
There exists a constant $C > 0$ such that for all $n \geq 1$:
\begin{enumerate} [label=(\roman*)]
  \item \label{l:Vn:decr} $V_n$ is non-increasing on $[0,\infty)$;
  \item $V_n$ and $V - V_n$ are convex on $(0,\infty)$;
  \item \label{l:Vn:supp} $\supp (V - V_n) \subset [-\tfrac1n, \tfrac1n]$;
  \item \label{l:Vn:0} $V_n(0) \leq C n^{a}$;
  \item \label{l:Vn:norm} For $f \in L^2(\R)$, $\displaystyle \| f \|_{V_n} := \sqrt{ \text{$\int_\R$} (V_n * f) f }$ defines a semi-norm;
  \item \label{l:Vn:Lp} $V_n \uparrow V$ in $L^p(\R)$ as $n \to \infty$ for any $1 \leq p < \frac1a$. 
\end{enumerate}
\end{lem}

\begin{proof}
Except for \ref{l:Vn:norm}, all properties are a direct consequence of the assumptions and properties of $V$ and the definition of $V_n$ in \eqref{Vn}. Property \ref{l:Vn:norm} can be proven along the lines of \cite[Lem.\ 3.2]{KimuraVanMeurs19acc}; it relies on the Fourier-transform of $V_n$ being non-negative, which easily follows from the other properties of $V_n$ (see \cite[Lem.\ 3.1]{KimuraVanMeurs19acc} for details).
\end{proof}

Next we establish an auxiliary estimate on $(V - V_n) * \orho$.

\begin{lem} \label{l:VVnastrho} 
There exists $C > 0$ such that for all $n \geq 1$ and all $0 \leq x \leq \frac12$
\begin{equation*}
  ((V - V_n) * \orho)(x) + ((V - V_n) * \orho)(1-x)
  \leq C n^{-\tfrac{1-a}2} \min \left\{ 1, [nx - 1]_+^{-\tfrac{1-a}2} \right\}.
\end{equation*}
\end{lem}

\begin{proof}
Take any $0 \leq x \leq \frac12$, and set $\underline x := (x- \frac1n) \vee 0$. By Lemma \ref{l:Vn}\eqref{l:Vn:supp} 
\begin{align*}
  ((V - V_n) * \orho)(x)
  = \int_{\underline x}^{x + \frac1n} (V - V_n) (x-y) \, \orho(y) \, dy.
\end{align*}
Since $V_n \geq 0$, we have that $(V - V_n)(r) \leq V(r) \leq C/|r|^a$ for all $|r| \leq 1$. Together with the upper bound on $\orho$ in Lemmas \ref{l:orho}\ref{l:orho:UB} and $\bar x \geq x - \frac1n$, we continue the estimate by
\begin{equation} \label{pf:VVn}
  \int_{\underline x}^{x + \frac1n} (V - V_n) (x-y) \, \orho(y) \, dy
  \leq C \int_{x - \frac1n}^{x + \frac1n} |x-y|^{-a} |y|^{-\tfrac{1-a}2} \, dy.
\end{equation}
Noting that $a > \frac2{1+a}$, we apply H\"older's inequality with $p \in (\frac2{1+a}, a)$ and conjugate exponent $q = p/(p-1)$. This yields
\[
  \int_{x - \frac1n}^{x + \frac1n} |x-y|^{-a} |y|^{-\tfrac{1-a}2} \, dy
  \leq \bigg( \int_{ - \frac1n}^{\frac1n} |y|^{-ap} \, dy \bigg)^{1/p} \bigg( \int_{x - \frac1n}^{x + \frac1n} |y|^{-\tfrac{1-a}2q} \, dy \bigg)^{1/q}.
\]
Since $ap < 1$ and $\frac{1-a}2 q < 1$ by construction, the right-hand side is finite. Noting that the right-hand side is maximal at $x=0$, we obtain
\[
  \bigg( \int_{ - \frac1n}^{\frac1n} |y|^{-ap} \, dy \bigg)^{1/p} \bigg( \int_{ - \frac1n}^{\frac1n} |y|^{-\tfrac{1-a}2q} \, dy \bigg)^{1/q}
  = C \Big( \frac1n \Big)^{\tfrac1p - a} \Big( \frac1n \Big)^{\tfrac1q - \tfrac{1-a}2}
  = C n^{- \tfrac{1-a}2}.
\]
In conclusion, the estimates above yield that
\[
  ((V - V_n) * \orho)(x) \leq C n^{- \tfrac{1-a}2}
\]
for all $0 \leq x \leq \frac12$.

To sharpen the bound for $\frac2n \leq x \leq \frac12$, we follow the estimate above until \eqref{pf:VVn}, and continue as follows:
\begin{align*}
  ((V - V_n) * \orho)(x)
  &\leq C \int_{x - \frac1n}^{x + \frac1n} |x-y|^{-a} |y|^{-\tfrac{1-a}2} \, dy \\
  &\leq C (x - \tfrac1n)^{-\tfrac{1-a}2} \int_{x - \frac1n}^{x + \frac1n} |x-y|^{-a} \, dy  
  \leq C' n^{-1+a} (x - \tfrac1n)^{-\tfrac{1-a}2}.
\end{align*}
The estimate for $((V - V_n) * \orho)(1-x)$ is analogous.
\end{proof}

\begin{proof}[Proof of Proposition \ref{prop:LB:En}]
The assertion of Proposition \ref{prop:LB:En} is obvious when $n$ is bounded from above by a fixed integer $N \geq 1$. In this case, it suffices to take $C = N^{1-a} E(\orho)$. Therefore, it is not restrictive to assume that $n \geq 1/b$, 
where $b > 0$ is as in \eqref{Vpp:LB}.

Let $\bx \in \Omega$ be given, and set 
\begin{equation*} 
  \mu_n := \frac1n \sum_{i=0}^n \delta_{x_i}
  \quad \text{and} \quad
  \nu_n := \mu_n - \orho.
\end{equation*}
Let
\begin{equation*}
  \Delta_1 := \{ (x_i, x_j) : |i-j| \leq 1 \} \subset \R^2
\end{equation*}
be the particle pairs that are left out in the interaction term of $E_n (\bx) - \NN (\bx)$, i.e.
\begin{equation*}
  E_n (\bx) - \NN (\bx)
  = \frac12 \iint_{\Delta_1^c} V(x-y) \, d\mu_n(y) d\mu_n(x) + \int U \, d\mu_n.
\end{equation*}
Then, we use Lemma \ref{l:orho}\ref{l:orho:EL} to estimate
\begin{align*} 
  &E_n (\bx) - \NN (\bx) - E ( \orho ) \\
  &= \frac12 \iint_{\Delta_1^c} V(x-y) \, d\nu_n(y) d\nu_n(x) 
    + \int (V * \orho) d\nu_n
    + \int U \, d\nu_n \\
  &\geq \frac12 \iint_{\Delta_1^c} V(x-y) \, d\nu_n(y) d\nu_n(x) \\
  &= \frac12 \underbrace{ \iint_{\Delta_1^c} V_n(x-y) \, d\nu_n(y) d\nu_n(x) }_{\Th} 
    + \frac12 \underbrace{ \iint_{\Delta_1^c} (V - V_n) (x-y) \, d\nu_n(y) d\nu_n(x) }_{\Ti},
\end{align*}
where $V_n$ is the regularisation introduced in \eqref{Vn}. To bound $\Th$, we use that $\max_\R V_n = V_n(0) \leq C n^a$ by Lemma \ref{l:Vn}\ref{l:Vn:decr},\ref{l:Vn:0} and that $\| \cdot \|_{V_n}^2$ is a norm (see Lemma \ref{l:Vn}\ref{l:Vn:norm}). Then, by the definition of $\Delta_1$, this yields
\begin{align*}
  \Th 
  = \| \nu_n \|_{V_n}^2 - \frac{n+1}{n^2} V_n(0) - \frac2{n^2} \sum_{i=1}^n V_n(x_i - x_{i-1})
  \geq - C n^{-1+a}.
\end{align*}

It remains to bound $\Ti$ from below by $-C n^{-1+a}$. We expand $\nu_n = \mu_n - \orho$ to rewrite
\begin{equation} \label{T2:bd1}
  \Ti 
  = \frac2{n^2} \sum_{i=2}^n \sum_{ j=0 }^{n-i} (V - V_n)(  x_{j+i} -  x_j ) 
       - 2\int \big((V - V_n) * \orho \big) d\mu_n
       + \iint_{\R^2} (V - V_n) (x-y) \, d\orho(y) d\orho(x).
\end{equation}
The third term is non-negative; we bound it from below by $0$. For the first two terms, we assume for convenience that $n$ is a multiple of $4$, and partition the interval $[0,1]$ into the closed intervals $I_k := [2 \frac{k-1}n, 2\frac kn]$ where $k = 1,\ldots,\frac n2$. Note that these intervals only overlap at their endpoints. Then, we remove the contribution of the interaction between any two particles located in different intervals $I_k$ from the double sum in the right-hand side of \eqref{T2:bd1}. Finally, we minimise the right-hand side of \eqref{T2:bd1} over each $I_k$ separately, and relax the constraint that the total number of particles should be $n+1$. This yields
\begin{equation} \label{T2:bd2}
  \Ti  
  \geq 2 \sum_{k=1}^{n/2} \min_{N \in \N} \bigg( \underbrace{ \frac1{n^2} \min_{2 \frac{k-1}n \leq y_1 \leq y_2 \leq \ldots \leq y_N \leq 2\frac kn} \sum_{i=2}^{ N } \sum_{ j= {1} }^{ {N}-i } (V - V_n)(  y_{j+i} -  y_j ) }_{ {T_{10}} } 
       - \frac Nn \| (V - V_n) * \orho \|_{C(I_k)} \bigg).
\end{equation}
We treat both terms within the parentheses separately. For the second term, we apply the bound in Lemma \ref{l:VVnastrho}. Since this bound gives the same estimate for the intervals $I_k$ and $I_{n/2 - k + 1}$, we focus on bounding it for $k \leq n/4$. This yields
\begin{align*} 
  \frac Nn \| (V - V_n) * \orho \|_{C(I_k)}
  &\leq C \frac Nn n^{-\tfrac{1-a}2} \min \left\{ 1, [2(k-1) - 1]_+^{-\tfrac{1-a}2} \right\} \\
  &\leq C' N n^{-1 - \tfrac{1-a}2} \left\{ \begin{array}{ll}
    1
    & k=1 \\
    (k-1)^{-\tfrac{1-a}2}
    & k \geq 2.
  \end{array} \right.
\end{align*} 

In particular, if the minimum over $N$ in \eqref{T2:bd2} is reached below an $n$-independent value $C$ (i.e., $N \leq C$), then it suffices to bound the first term in parentheses in \eqref{T2:bd2} from below simply by $0$. Therefore, we assume next that the minimiser $N$ is sufficiently large; in particular $N \geq 9$. We further assume for simplicity that $N$ is a multiple of $3$.

To bound the term $T_{10}$ in \eqref{T2:bd2}, we rely on the basic arguments in the theory of $i$-th neighbour interaction energies with convex interaction potentials. In more detail, first we use that the summands are non-negative to remove the latter terms in the sum over $i$:
\begin{equation*}
  T_{10} \geq \frac1{n^2} \min_{2 \frac{k-1}n \leq y_1 \leq y_2 \leq \ldots \leq y_N \leq 2\frac kn} \sum_{i=2}^{ N/3 } \sum_{ j=1 }^{ {N}-i } (V - V_n)(  y_{j+i} -  y_j ).
\end{equation*}
Then, we bound the minimum from below by exchanging the sum over $i$ with the minimisation over $(y_j)_j$, i.e.,
\begin{equation*}
  T_{10} \geq \frac1{n^2} \sum_{i=2}^{ N/3 } \min_{2 \frac{k-1}n \leq y_1 \leq y_2 \leq \ldots \leq y_N \leq 2\frac kn} \sum_{ j=1 }^{ {N}-i } (V - V_n)(  y_{j+i} -  y_j ) =: T_{11}.
\end{equation*}
Then, the resulting minimisation problem can be written as a sum over independent minimisation problems. To see this, we change variables in the summation index by $j = (m-1)i + \ell$, and write the sum over $j$ as two sums over $\ell$ and $m$. By possibly skipping a few terms for those $j$ that are close to $N-i$, we obtain
\begin{align*}
  T_{11} 
  &\geq \frac1{n^2} \sum_{i=2}^{ N/3 } \min_{2 \frac{k-1}n \leq y_1 \leq y_2 \leq \ldots \leq y_N \leq 2\frac kn} \sum_{ \ell=1 }^{ i } \sum_{m=1}^{ \lfloor  \frac{N-\ell}{i} \rfloor } (V - V_n)(  y_{mi + \ell} -  y_{(m-1)i + \ell} ) \\
  &\geq \frac1{n^2} \sum_{i=2}^{ N/3 } \sum_{ \ell=1 }^{ i } \min_{2 \frac{k-1}n \leq y_1 \leq y_2 \leq \ldots \leq y_N \leq 2\frac kn} \sum_{m=1}^{ \lfloor  \frac{N-\ell}{i} \rfloor } (V - V_n)(  y_{mi + \ell} -  y_{(m-1)i + \ell} ) \\
  &= \frac1{n^2} \sum_{i=2}^{ N/3 } \sum_{ \ell=1 }^{ i } \min_{2 \frac{k-1}n \leq y_{\ell} \leq y_{i + \ell} \leq y_{2i + \ell} \leq \ldots \leq y_{\lfloor  \frac{N-\ell}{i} \rfloor i + \ell} \leq 2\frac kn} \sum_{m=1}^{ \lfloor  \frac{N-\ell}{i} \rfloor } (V - V_n)(  y_{mi + \ell} -  y_{(m-1)i + \ell} ) \\
  &= \frac1{n^2} \sum_{i=2}^{ N/3 } \sum_{ \ell=1 }^{ i } \min_{2 \frac{k-1}n \leq z_0 \leq z_1 \leq \ldots \leq z_{\lfloor  \frac{N-\ell}{i} \rfloor} \leq 2\frac kn} \sum_{m=1}^{ \lfloor  \frac{N-\ell}{i} \rfloor } (V - V_n)(  z_{m} -  z_{m-1} ),
\end{align*}
where in the last equality we change to the variable $z_m := y_{mi + \ell}$.
Each such minimisation problem over $z_m$ involves only nearest neighbour interactions with the convex, repelling interaction potential $V - V_n$, which is minimised by the equispaced configuration. Plugging in the equispaced configuration $z_m = 2 \frac{k-1}n + 2 m/(n \lfloor  \frac{N-\ell}{i} \rfloor )$, we get
\begin{align} \notag
  T_{11}
  &\geq \frac1{n^2} \sum_{i=2}^{{N/3}} \sum_{\ell = 1}^i \lfloor \tfrac{N-\ell}i \rfloor (V - V_n) \bigg(  \frac2{ n \lfloor \tfrac{N-\ell}i \rfloor } \bigg) 
    \geq \frac1{n^2} \sum_{i=2}^{N/3} \sum_{\ell = 1}^i \Big( \frac23 \frac{N}i - 1 \Big) (V - V_n) \Big(  \frac{3i}{nN} \Big) \\\notag
    &= \frac1{n^2} \sum_{i=2}^{N/3} ( \tfrac23 N - i ) (V - V_n) \Big(  \frac{3i}{nN} \Big) 
    \geq \frac{N}{3 n^2} \sum_{i=2}^{N/3} (V - V_n) \Big(  \frac{3i}{nN} \Big) \\ \label{T2:bd3}
    &\geq \frac{N^2}{9 n} \int_{\frac{6}{nN}}^{\frac1n} (V - V_n) (x) \, dx
    = \frac{N^2}{9 n^2} \int_{\frac{6}{N}}^{1} (V - V_n) (\tfrac xn) \, dx,
\end{align}
where in the last inequality we have recognized the sum as a Riemann upper-sum. To estimate the integrand from below, we integrate twice, use that $(V - V_n)(\frac1n) = 0$ and rely on $V_n'' = 0$ on $(0, \frac1n)$ and the lower bound on $V''$ on $(0,b) \supset (0,\frac1n)$ in \eqref{Vpp:LB} to deduce that
\begin{align*}
  (V - V_n) \Big( \frac xn \Big)
  = \int_{\frac xn}^{\frac1n} \int_y^{\frac1n} V''(z) \, dz dy
  \geq C \int_{\frac xn}^{\frac1n} \int_y^{\frac1n} z^{-2-a} \, dz dy
  = C n^a \int_x^1 \int_y^1 z^{-2-a} \, dz dy
\end{align*}
for all $0 < x < 1$. The double integral in the right-hand side is independent of $n$, positive, and decreasing as a function of $x$. Using this and noting that $6/N \leq 2/3$, we continue the estimate in \eqref{T2:bd3} by
\begin{align*}
 \frac{N^2}{9 n^2} \int_{\frac{6}{N}}^{1} (V - V_n) (\tfrac xn) \, dx
 \geq C N^2 n^{-2+a} \int_{\frac23}^{1} \int_x^1 \int_y^1 z^{-2-a} \, dz dy dx 
 \geq C' N^2 n^{-2+a}.
\end{align*}

Finally, collecting our estimates in \eqref{T2:bd2}, we obtain two constants $C, C' > 0$ such that
\begin{align*}
  \Ti 
  &\geq C' \sum_{k=1}^{n/2} \min_{N \in \R} \Big( N^2 n^{-2+a} - C k^{-\tfrac{1-a}2} N n^{-1 - \tfrac{1-a}2} \Big) \\
  &= C' n^{-2+a} \sum_{k=1}^{n/2} \min_{N \in \R} N \bigg( N - C \Big(\frac kn \Big)^{-\tfrac{1-a}2} \bigg) \\
  &= - C' n^{-2+a} \frac{C^2}4 \sum_{k=1}^{n/2} \Big(\frac kn \Big)^{-1+a} 
  \geq - C n^{-1+a}.
\end{align*}
\end{proof}

\begin{rem}[The case $\orho \leq C$] \label{r:prop:LB:En}
 When $\orho$ is bounded, the proof of Proposition \ref{prop:LB:En} simplifies significantly. Indeed, if $\orho$ is bounded, then instead of Lemma \ref{l:VVnastrho} the rougher estimate $\| (V - V_n) * \orho \|_{C([0,1])} \leq C n^{-1+a}$ is sufficient, because this estimate gives immediately the desired bound on the second term in \eqref{T2:bd1}.
\end{rem}

\section{Proof of Theorem \ref{t}}
\label{s:a0}

We first treat the case $0 < a < 1$. Given $V$ and $U$ as in Theorem \ref{t}, we start by constructing a more convenient interaction potential $\tilde V$ whose corresponding energies $\tilde E$ and $\tilde E_n$ have the same sets of minimisers as $E$ and $E_n$ respectively. Let $R$ and $S$ be as in Proposition \ref{prop:VR}. Let $\tilde V$ satisfy Assumption \ref{ass:VU} such that
\begin{equation*}
  \tilde V = V \text{ on } \big( 0, \max \{2S, R\} \big]
  \quad \text{and} \quad
  \tilde V = C \text{ on } \big[ \max \{2S, R\} + 1, \infty \big)
\end{equation*}
for some constant $C \in \R$. Such a $\tilde V$ can be obtained by multiplying $V'(x)$ with a cut-off function, and then integrating from $x$ to $\infty$. From the integrability and convexity of $\tilde V$ we note that $C = \min_\R \tilde V$. Then, Proposition \ref{prop:VR} applies to $\tilde V$ with the same constants $R$ and $S$, and thus any minimiser of 
\begin{equation*}
  \tilde E (\rho) := \frac12 \int_\R (\tilde V * \rho) \, d\rho + \int_\R U \, d\rho
\end{equation*}
is also supported in $[-S,S]$. Since by the choice of $\tilde V$ it holds that $\tilde E = E$ on $\cP ([-S,S])$, any minimiser of $E$ is a minimiser of $\tilde E$ and vice versa. Analogously, we obtain the same conclusion for $E_n$ for a possibly different constant $S$. Hence, we may replace $V$ in Theorem \ref{t} by $\tilde V$. In addition, we may further subtract the constant $C/2 = \min_\R \tilde V / 2$ from $\tilde E$ and subtract the constant $(1 + \frac1n)C/2$ from $\tilde E_n$ such that the resulting interaction potential
has compact support. We denote this potential by $\tilde V$ without changing notation.

The existence of minimisers of $E_n$ is shown in Remark \ref{r:En:minz:ex}. Since $\tilde V$ has compact support, Lemma \ref{l:orho} applies. This shows that $\tilde E$ has a unique minimiser $\tilde \rho$. Moreover, by Lemma \ref{l:orho}\ref{l:orho:supp} there exists an affine change of variables such that $\supp \tilde \rho$ turns into the interval $[0,1]$. It is not difficult to verify (see \cite[Step 1 in the proof of Lem.~6.8]{KimuraVanMeurs19acc}) that under this change of variables, $E$ and $E_n$ (possibly multiplied or shifted by a constant) still satisfy Assumption \ref{ass:VU}, and that the resulting Sobolev norm \eqref{Hs} is equivalent to that before the change of variables. Hence, redefining $V, U, E, E_n, \orho, \tx$ as the potentials, energies and minimisers that appear after the affine change of variables has been applied, we observe that $V,U$ satisfy Assumption \ref{ass:VU}, that the additional property \eqref{suppV:suppOrho} holds, and that $\| \tphi - \overline \rho \|_{H^{-(1-a)/2} (\R)}^2$ has changed by an $n$-independent multiplicative constant with respect to the original value.

The remaining part of the proof concerns the estimate of $\| \tphi - \overline \rho \|_{H^{-(1-a)/2} (\R)}^2$. This part is already sufficiently detailed in the sketch of the proof given in Section \ref{s:intro:t}; see \eqref{for:conv:rate:DtC:first:IDy:est}--\eqref{En:LB} and the references therein to Sections \ref{s:Part2} and \ref{s:Part3}. This completes the proof of Theorem \ref{t} in the case $0 < a < 1$.
\medskip

Finally, we treat the case $a = 0$. The proof is analogous to the case $0 < a < 1$; the only differences are several minor changes in the computations. All these changes are ramifications of the change in the upper bound on $V$, which is
\begin{equation*}
  V(r) \leq C - \log |r|.
\end{equation*}
In Table \ref{tab:1} and in the two items listed below we mention all statements of Sections \ref{s:Part1} -- \ref{s:Part3} which are not literally valid for $a=0$, and provide the required modification.

\begin{table}[h]
\centering
\begin{tabular}{ll} 
  \toprule
  Statement & updated estimate \\
  \midrule
  Lemma \ref{l:ox}\ref{l:ox:NN} & $D_n(\ox) + \NN (\ox) \leq C n^{-1} \log n$ \\
  Lemma \ref{l:ophi:orho:V} & $\| \overline \varphi  \|_V^2 - \| \overline \rho \|_V^2 \leq C n^{-1} \log n$ \\
  Lemma \ref{l:Vn}\ref{l:Vn:0} & $V_n(0) \leq \log n + C$ \\
  Lemma \ref{l:Vn}\ref{l:Vn:Lp} & $V_n \uparrow V$ in $L^p(\R)$ as $n \to \infty$ for any $1 \leq p < \infty$ \\
  Lemma \ref{l:VVnastrho} & $\left\{ \begin{array}{r}
  ((V - V_n) * \orho)(x) + ((V - V_n) * \orho)(1-x) \\
  \leq C n^{-1/2} (\log n) \min \left\{ 1, [nx - 1]_+^{-1/2} \right\}
  \end{array} \right.$ \\
  Proposition \ref{prop:LB:En} & $E_n (\bx) - \NN (\bx) - E ( \orho ) 
  \geq -C n^{-1} (\log n)^3$ \\
  \bottomrule
\end{tabular}
\caption{Changes in the estimates for $a = 0$.}
\label{tab:1}
\end{table}

\begin{enumerate}
  \item The constant $C'$ in Lemma \ref{l:Dn:by:NN} also contains a contribution from $V_a$. 
  \item In the proof of Proposition \ref{prop:LB:En}, the estimate in the final display changes as follows:
  \begin{align*}
  \Ti 
  &\geq C' \sum_{k=1}^{n/2} \min_{N \in \R} \Big( N^2 n^{-2} - C k^{-1/2} N n^{-3/2} \log n \Big) \\
  &= C' n^{-2} \sum_{k=1}^{n/2} \min_{N \in \R} N \bigg( N - C (\log n) \Big(\frac kn \Big)^{-1/2} \bigg) \\
  &= - \frac{C'}n (\log n)^2 \frac{C^2}4 \frac 1n \sum_{k=1}^{n/2} \Big(\frac kn \Big)^{-1} 
  \geq - C n^{-1} (\log n)^3.
\end{align*}
\end{enumerate}

\section{Numerical computations on the rate in Theorem \ref{t}}
\label{s:num}

The aim of this section is to compare the upper bound of the convergence rate in Theorem \ref{t} with the actual convergence rate in concrete examples. These concrete examples are given by specific choices for the potentials $V$ and $U$ for which all quantities except for $\tx$ can be computed explicitly. With this aim, we take $\Vreg = 0$ and $U$ a convex polynomial on $D(U)$. Given the qualitatively different profiles of $\orho$ observed in Figure \ref{fig:tx}, we consider two choices for $D(U)$; a bounded interval (Case 1) and $\R$ (Case 2).

For each of these two cases, the method to test Theorem \ref{t} numerically is as follows. First, we compute $\tx$ by minimizing $E_n$ in \eqref{En} numerically with Newton's method for several values of $n$. One observation we did from this data is that
\begin{equation} \label{xi:in:supprho}
  x_i^* \in \supp \orho 
  \quad \text{for all } i = 0,1,\ldots,n
\end{equation}
for each value of $n$ used in our simulations.

Then, instead of using the norm in $H^{-(1-a)/2}$, we use the equivalent norm $\| \cdot \|_V$ to make the computation easier. Indeed, since by \eqref{xi:in:supprho} and Lemma \ref{l:orho}\ref{l:orho:EL} the inequality in \eqref{for:conv:rate:DtC:first:IDy:est} becomes an equality, we obtain that
\begin{equation} \label{en}
  e_n 
  := \| \overline \rho - \tphi \|_V^2
  = 2 \big( E(\tphi) - E(\orho) \big).
\end{equation}
Now, $E(\orho)$ can be computed explicitly given that $\Vreg = 0$ and $U$ is a polynomial. To compute $E(\tphi)$, we set $\bx := \tx$ and $\ell_i := x_i - x_{i-1}$, and obtain from \eqref{Ephi:nice} that
  \begin{equation} \label{EphiS}
    E ( \tphi ) 
  = \frac1{2 n^2} \sum_{i=1}^n \sum_{ j=1 }^n \frac1{ \ell_i \ell_j } \int_{x_{i-1}}^{x_i} \int_{x_{j-1}}^{x_j} V_a (x - y)\, dy dx
    + \frac1n \sum_{i=1}^n \frac1{\ell_i} \int_{x_{i-1}}^{x_i} U(x) \, dx.
  \end{equation}
Since $U$ is a polynomial, both integrals above can be computed explicitly as a function of $\bx$. Hence, once $\tx$ is computed numerically, $e_n$ can be computed without any further numerical error (except for machine precision).

Finally, to compare the numerically computed values for $e_n$ with Theorem \ref{t}, we make the ansatz 
\begin{equation*} 
  e_n = C n^{-p}.
\end{equation*}
Then, $e_n / e_{2n} = 2^p$, and thus 
\begin{equation} \label{p}
  p = \frac{ \log e_n - \log {e_{2n}} }{ \log 2 }.
\end{equation}
Hence, by taking $n$ as subsequent powers of $2$, we can compute $p$ for each pair of subsequent values of $e_n$, and compare the values of $p$ with the theoretically obtained power $1-a$. 
 
\paragraph{Case 1: the bounded domain $D(U) = [0,1]$.}
We take $D(U) = [0,1]$ and $U = 0$ on $[0,1]$. 
Following the computations in, e.g., \cite{KimuraVanMeurs19acc}, we obtain
\begin{equation*}
  \orho (x) 
  = \left\{ \begin{aligned}
    \frac1\pi \big[ x(1-x) \big]^{ - \tfrac12 }
    &\quad \text{if } a = 0  \\
    \frac{ a \Gamma(a) }{ \Gamma (\tfrac{1+a}2)^2 } \big[ x(1-x) \big]^{ - \tfrac{1-a}2 }
    &\quad \text{if } 0 < a < 1
  \end{aligned} \right.
\end{equation*}
and
\begin{equation*}
  E(\orho)
  = \left\{ \begin{aligned}
    \log 2
    &\quad \text{if } a = 0  \\
    \frac{ \pi a \Gamma(a) }{ 2 \Gamma (\tfrac{1+a}2)^2 \cos ( \tfrac{a\pi}2 ) }
    &\quad \text{if } 0 < a < 1,
  \end{aligned} \right.
\end{equation*}
where $\Gamma(\alpha) = \int_0^\infty x^{\alpha-1} e^{-x} \, dx$ is the usual $\Gamma$-function.

With $E(\orho)$ specified, we compute $e_n$ and $p$ in \eqref{en} and \eqref{p} with the method described above. The results are shown in Table \ref{tab:p} and Figures \ref{fig:en} and \ref{fig:p}. We note that $-p$ is the slope of the graphs of $e_n$ in Figure \ref{fig:en}. For all four values of $a$, $e_n$ seems to converge to $0$ as $n \to \infty$. Also, $p$ decreases as $a$ increases. These observations are in line with Theorem \ref{t}. However, for all four values of $a$, the computed value of $p$ is significantly larger than the theoretical prediction $1-a$ from Theorem \ref{t}. 

\begin{figure}[h]
\centering
\begin{tikzpicture}[scale=.889]
\def \x {1.33}
\def \dx {.13}
\def \y {1.3}
\begin{scope}[shift={(0,0)},scale=1]
  \node (label) at (0,0){\includegraphics[height=4cm, trim=0mm 0mm 0mm 0mm]{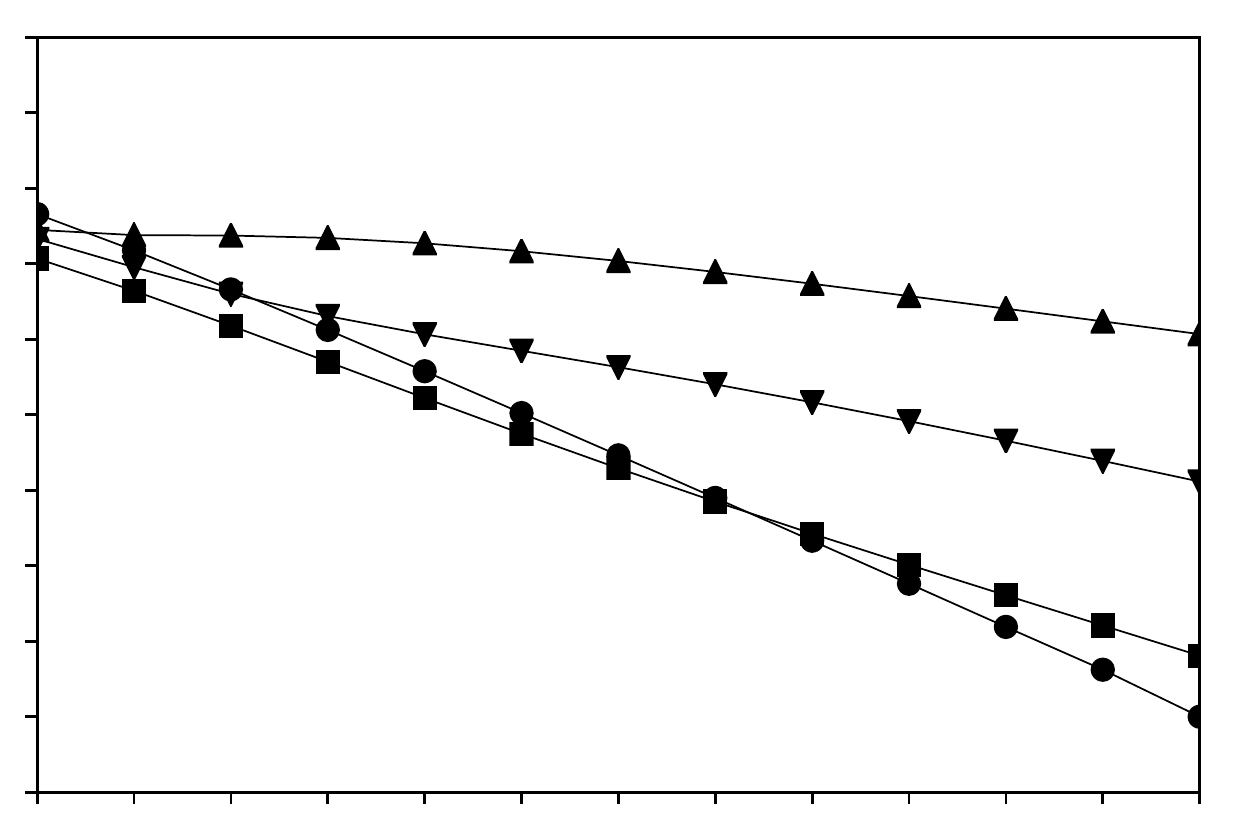}};
  \draw (-3.2, 1.65) node[left] {$10^0$};
  \draw (-3.2, .05) node[left] {$10^{-4}$};
  \draw (-3.2, -1.6) node[left] {$10^{-8}$};
  \draw (-3.2, 2.1) node[above] {$e_n$};
  \draw (-3.2, -2.1) node[below] {$2^2$};
  \draw (-1.07, -2.1) node[below] {$2^6$};
  \draw (1.07, -2.1) node[below] {$2^{10}$};
  \draw (3.2, -2.1) node[below] {$2^{14}$};
  \draw (3.2, -2.1) node[anchor = south west] {$n$};
  \draw (0,2.1) node[above] {Case 1};
\end{scope}
\begin{scope}[shift={(9,0)},scale=1]
  \node (label) at (0,0){\includegraphics[height=4cm, trim=0mm 0mm 0mm 0mm]{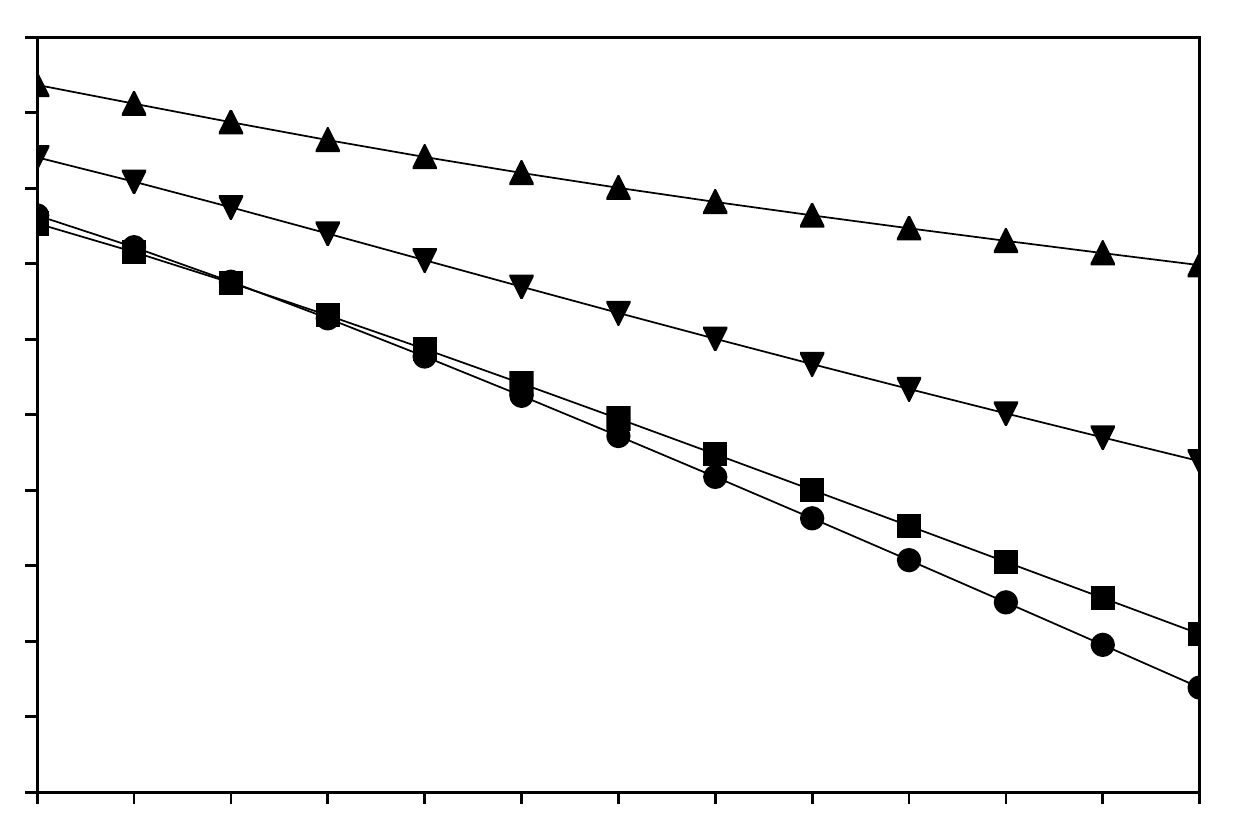}};
  \draw (-3.2, 1.65) node[left] {$10^0$};
  \draw (-3.2, .05) node[left] {$10^{-4}$};
  \draw (-3.2, -1.6) node[left] {$10^{-8}$};
  \draw (-3.2, 2.1) node[above] {$e_n$};
  \draw (-3.2, -2.1) node[below] {$2^2$};
  \draw (-1.07, -2.1) node[below] {$2^6$};
  \draw (1.07, -2.1) node[below] {$2^{10}$};
  \draw (3.2, -2.1) node[below] {$2^{14}$};
  \draw (3.2, -2.1) node[anchor = south west] {$n$};
  \draw (0,2.1) node[above] {Case 2};
\end{scope}
\end{tikzpicture} \\
\caption{The numerically computed values for $e_n$ (see \eqref{en}) in Cases 1 and 2 for the values $a = 0$ ($\bullet$), $a = \frac14$ ({\footnotesize $\blacksquare$}), $a = \frac12$ ($\blacktriangledown$) and $a = \frac34$ ($\blacktriangle$). }
\label{fig:en}
\end{figure}
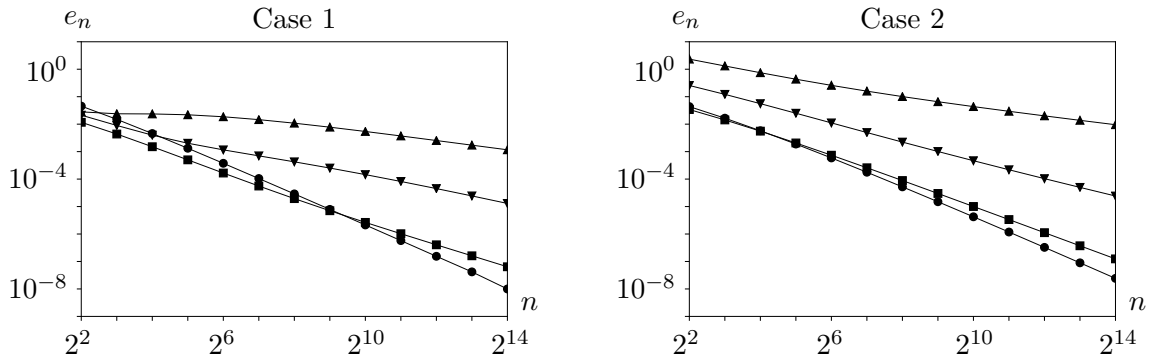

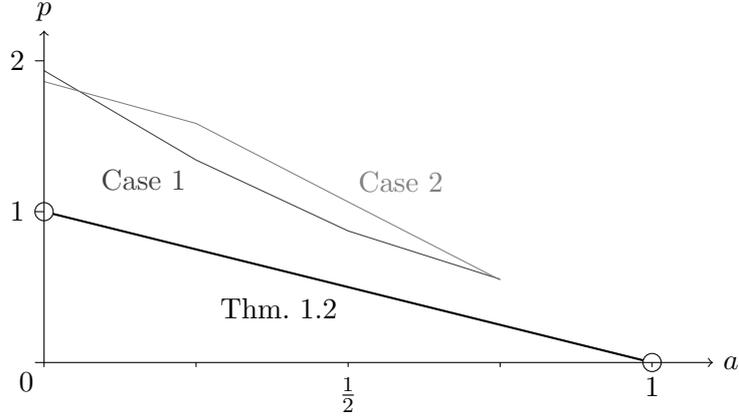
\begin{figure}[h]
\centering
\begin{tikzpicture}[xscale=8, yscale=4]    
\def \w {.015}

\draw[thick] (0,.5) -- (1,0) node[midway, anchor = north east]{Thm.\ \ref{t}};
\begin{scope}[shift={(0,.5)},yscale=2]
  \filldraw[fill = white] (0,0) circle (\w);
\end{scope}
\begin{scope}[shift={(1,0)},yscale=2]
  \filldraw[fill = white] (0,0) circle (\w);
\end{scope}

\draw[black!75] (0, .9675) -- (.25, .67125) node[anchor = north east] {Case 1} -- (.5, .43625) -- (.75, .27625);
\draw[black!50] (0, .93125) -- (.25, .7925) -- (.5, .5325) node[anchor = south west] {Case 2} -- (.75, .27375);
        
\draw[->] (0,-\w) -- (0,1.1) node[above]{$p$};
\draw[->] (-\w,0) -- (1.1,0) node[right]{$a$};
\draw (0,0) node[anchor = north east]{$0$};

\foreach \x in {.25, .5, .75} {
  \draw (\x, 0) --++ (0, -\w);
}
\draw (.5, -\w) node[below] {$\frac12$};
\draw (1, 0) --++ (0, -\w) node[below] {$1$};

\draw (0, .5) --++ (-\w, 0) node[left] {$1$};
\draw (0, 1) --++ (-\w, 0) node[left] {$2$};
\end{tikzpicture} \\
\caption{Values of $p$ as a function of $a$ in Cases 1 and 2 compared with the theoretical prediction from Theorem \ref{t}. The $n$-dependence is removed by taking the average of $p$ over the last four values of $n$ in Table \ref{tab:p}.}
\label{fig:p}
\end{figure}  

\begin{table}[h]
\centering 
\begin{tabular}{c|cccc}
\multicolumn{5}{c}{Values of $p$ in Case 1} \\
\toprule
$n$ & $a = 0$ & $a = \frac14$ & $a = \frac12$ & $a = \frac34$ \\
\midrule
$2^{2}$ & $1.60$ & $1.42$ & $1.24$ & $0.23$ \\ 
$2^{3}$ & $1.72$ & $1.54$ & $1.17$ & $0.02$ \\ 
$2^{4}$ & $1.78$ & $1.59$ & $0.98$ & $0.10$ \\ 
$2^{5}$ & $1.82$ & $1.60$ & $0.80$ & $0.24$ \\ 
$2^{6}$ & $1.84$ & $1.57$ & $0.72$ & $0.35$ \\ 
$2^{7}$ & $1.86$ & $1.52$ & $0.72$ & $0.43$ \\ 
$2^{8}$ & $1.87$ & $1.47$ & $0.75$ & $0.48$ \\ 
$2^{9}$ & $1.88$ & $1.41$ & $0.79$ & $0.52$ \\ 
$2^{10}$ & $1.89$ & $1.37$ & $0.83$ & $0.54$ \\ 
$2^{11}$ & $1.90$ & $1.34$ & $0.86$ & $0.55$ \\ 
$2^{12}$ & $1.88$ & $1.33$ & $0.89$ & $0.56$ \\ 
$2^{13}$ & $2.07$ & $1.33$ & $0.91$ & $0.56$ \\ 
\midrule
Thm. & $\approx 1$ & $0.75$ & $0.50$ & $0.25$ \\
\bottomrule 
\end{tabular}
\qquad
\begin{tabular}{c|cccc}
\multicolumn{5}{c}{Values of $p$ in Case 2} \\
\toprule
$n$ & $a = 0$ & $a = \frac14$ & $a = \frac12$ & $a = \frac34$ \\
\midrule
$2^{2}$ & $1.40$ & $1.26$ & $1.07$ & $0.83$ \\ 
$2^{3}$ & $1.52$ & $1.36$ & $1.13$ & $0.82$ \\ 
$2^{4}$ & $1.61$ & $1.43$ & $1.16$ & $0.79$ \\ 
$2^{5}$ & $1.68$ & $1.48$ & $1.17$ & $0.75$ \\ 
$2^{6}$ & $1.73$ & $1.52$ & $1.17$ & $0.70$ \\ 
$2^{7}$ & $1.77$ & $1.55$ & $1.15$ & $0.66$ \\ 
$2^{8}$ & $1.80$ & $1.56$ & $1.14$ & $0.62$ \\ 
$2^{9}$ & $1.82$ & $1.57$ & $1.12$ & $0.59$ \\ 
$2^{10}$ & $1.84$ & $1.58$ & $1.09$ & $0.57$ \\ 
$2^{11}$ & $1.86$ & $1.59$ & $1.07$ & $0.55$ \\ 
$2^{12}$ & $1.87$ & $1.59$ & $1.06$ & $0.54$ \\ 
$2^{13}$ & $1.88$ & $1.58$ & $1.04$ & $0.53$ \\ 
\midrule
Thm. & $\approx 1$ & $0.75$ & $0.50$ & $0.25$ \\
\bottomrule 
\end{tabular}
\caption{The numerically computed values for $p$ (see \eqref{p}) in Cases 1 and 2. The bottom row is the prediction from Theorem \ref{t}. In Case 1 with $a = 0$ the last two values of $p$ seem off; we expect that this is due to the numerical rounding errors that were made when \eqref{EphiS} was computed. } 
\label{tab:p}
\end{table}

\paragraph{Case 2: the infinite domain $D(U) = \R$.}

We take $D(U) = \R$ and 
$$U(x) = \gamma_a \Big(x - \frac12 \Big)^2,
\qquad \gamma_a := \left\{ \begin{aligned}
    4
    &\quad \text{if } a = 0  \\
    \frac{ 2 \pi a^2 (2+a) \Gamma(a) }{ \Gamma (\tfrac{1+a}2)^2 \cos ( \tfrac{a\pi}2 ) }
    &\quad \text{if } 0 < a < 1.
  \end{aligned} \right.$$
The constant $\gamma_a$ is chosen such that $\supp \orho = [0,1]$. Following the computations in, e.g., \cite{KimuraVanMeurs19acc} and \cite[Chap.\ IV, Thm.\ 5.1]{SaffTotik97}, we obtain
\begin{equation*}
  \orho (x) 
  = \left\{ \begin{aligned}
    \frac8\pi \big[ x(1-x) \big]^{ \tfrac12 }
    &\quad \text{if } a = 0  \\
    \frac{ 4 (2+a) a \Gamma(a) }{ (1+a) \Gamma (\tfrac{1+a}2)^2 } \big[ x(1-x) \big]^{ \tfrac{1+a}2 }
    &\quad \text{if } 0 < a < 1
  \end{aligned} \right.
\end{equation*}
and
\begin{equation*}
  E(\orho)
  = \left\{ \begin{aligned}
    \log 2
    &\quad \text{if } a = 0  \\
    \frac{ \pi (2+a)^2 a \Gamma(a) }{ 2 (4+a) \Gamma (\tfrac{1+a}2)^2 \cos ( \tfrac{a\pi}2 ) }
    &\quad \text{if } 0 < a < 1.
  \end{aligned} \right.
\end{equation*}

Similar to Case 1, we compute $e_n$ and $p$. The results are shown in Figure \ref{fig:en} and Table \ref{tab:p}. The similarities with Case 1 are that $e_n$ seems to converge to $0$ as $n \to \infty$, that $p$ decreases as $a$ increases, and that the computed value of $p$ is significantly larger than the theoretical prediction $1-a$ from Theorem \ref{t}. 
\smallskip

We end this section with three quantitative comparisons between Cases 1 and 2:
\begin{itemize}
  \item For $a = 0$, the values of $e_n$ and $p$ are similar. 
  \item When $a$ increases, the values of $e_n$ are larger in Case 2 than in Case 1 (at least when $n$ is not too large). This is consistent with Figure \ref{fig:tx}, where the graph of $\tphi$ seems a better match with the graph of $\orho$ in Case 1 than in Case 2.
  \item Yet, the values of $p$ are larger in Case 2, which would imply that for $n$ large enough, the values of $e_n$ in Case 2 are smaller than those in Case 1. A possible reason for this could be the singularities of $\orho$ at $x = 0$ and $x = 1$ in Case 1. 
\end{itemize}

\paragraph{Acknowledgments}
MK gratefully acknowledge support from JSPS KAKENHI Grant Number 17H02857.

PvM gratefully acknowledges support from the International Research Fellowship of the Japanese Society for the Promotion of Science and the associated JSPS KAKENHI Grant Number 15F15019.

\end{document}